\documentclass[11pt]{article}
\usepackage[utf8]{inputenc}
\usepackage[english]{babel}

\usepackage{amsfonts}
\usepackage{amsmath,amsthm,amscd,amssymb,mathrsfs,setspace}
\usepackage{mathtools}
\usepackage{cite}
\usepackage{latexsym,epsf,epsfig}
\usepackage{color}
\usepackage[a4paper,top=2cm,bottom=2cm,left=3cm,right=3cm,marginparwidth=1.75cm]{geometry}
\usepackage[colorlinks]{hyperref}
\usepackage[nameinlink,capitalise]{cleveref}
\usepackage[mathscr]{euscript}
\usepackage{dutchcal}
\usepackage{algorithm}
\usepackage{algorithmicx}
\usepackage{algpseudocode}
\usepackage[table]{xcolor}
\usepackage{soul}
\usepackage{subcaption}
\usepackage{siunitx}
\usepackage{tikz}
\usetikzlibrary{decorations.pathreplacing,calligraphy}

\usepackage{booktabs}
\usepackage{multirow}
\usepackage{adjustbox}
\usepackage{authblk}

\definecolor{lccx}{HTML}{92268F}

\ifpdf
\hypersetup{
	pdftitle={McC 4 MIPDECOs},
	pdfauthor={S. Leyffer and P. Manns},
	linkcolor=lccx
}
\fi

\newcommand{\N}{\mathbb{N}}
\newcommand{\R}{\mathbb{R}}
\newcommand{\Z}{\mathbb{Z}}

\newcommand{\calQ}{\mathcal{Q}}
\newcommand{\calF}{\mathcal{F}}

\DeclareMathOperator*{\BV}{BV}

\DeclareMathOperator*{\TV}{TV}
\DeclareMathOperator*{\TVh}{TV{}_h}
\DeclareMathOperator*{\dd}{d}

\DeclareMathOperator*{\dvg}{div}
\newcommand{\weakto}{\rightharpoonup}
\newcommand{\weakstarto}{\stackrel{\ast}{\rightharpoonup}}

\newtheorem{theorem}{Theorem}[section]
\newtheorem{proposition}[theorem]{Proposition}
\newtheorem{lemma}[theorem]{Lemma}

\theoremstyle{definition}

\newtheorem{assumption}[theorem]{Assumption}

\theoremstyle{remark}
\newtheorem{remark}[theorem]{Remark}

\crefname{assumption}{Assumption}{Assumptions}
\Crefname{assumption}{Assumption}{Assumptions}

\title{McCormick envelopes in mixed-integer PDE-constrained optimization}
\author[1]{Sven Leyffer}
\author[2]{Paul Manns}
\affil[1]{Mathematics and Computer Science Division, Argonne National Laboratory, Lemont, IL 60439, USA, \textit{leyffer@anl.gov}}
\affil[2]{Faculty of Mathematics, TU Dortmund University, 44227 Dortmund, Germany, \textit{paul.manns@tu-dortmund.de}}

\begin{document}
\maketitle
\begin{abstract}
McCormick envelopes are a standard tool for deriving convex relaxations of optimization problems that involve polynomial terms. 
Such McCormick relaxations provide lower bounds, for example, in branch-and-bound procedures for mixed-integer
nonlinear programs but have not gained much attention in PDE-constrained optimization so far.
This lack of attention may be due to the distributed nature of such problems, which on the one hand
leads to infinitely many linear constraints (generally state constraints that may be difficult to handle)
in addition to the state equation for a pointwise formulation of the McCormick envelopes
and renders bound-tightening procedures that successively improve the resulting convex relaxations computationally intractable.

We analyze McCormick envelopes for a model problem class that is governed by a semilinear PDE involving a bilinearity
and integrality constraints. We approximate the nonlinearity and in turn the McCormick envelopes by averaging the involved
terms over the cells of a partition of the computational domain on which the PDE is defined.
This yields convex relaxations that underestimate the original problem up to an a priori error estimate
that depends on the mesh size of the discretization. These approximate McCormick relaxations can be improved
by means of an optimization-based bound-tightening procedure. We show that their minimizers converge
to minimizers to a limit problem with a pointwise formulation of the McCormick envelopes when driving the mesh size to zero.

We provide a computational example, for which we certify all of our imposed assumptions. The results
point to both the potential of the methodology and the gaps in the research that need to be closed.

Our methodology provides a framework  first for obtaining pointwise underestimators for nonconvexities
and second for approximating them with finitely many linear inequalities in an infinite-dimensional
setting.
\end{abstract}

\section{Introduction}\label{sec:intro}
We are interested in the global optimization of mixed-integer PDE-constrained optimization
problems (MIPDECOs) that feature nonconvex terms. MIPDECOs arise in many real-world applications
such as topology optimization \cite{haslinger2015topology,leyffer2021convergence,bendsoe2013topology}
and supply network optimization \cite{pfetsch2015validation,d2015mathematical}.
A prototypical problem class is
\begin{gather}\label{eq:mipdeco}
\begin{aligned}
\min_{u, w}\ & j(u, w) + \alpha R(w) \\
\text{ s.t.\ } & Au + N(u,w) = 0\\
& w \in C \\
& w(x) \in \Z \text{ for almost every (a.e.) } x \in \Omega,
\end{aligned}\tag{MIPDECO}
\end{gather}
where $w : \Omega \to \R$ is a measurable input function (generally also called \emph{control}) on a
computational domain $\Omega \subset \R^d$, $d \in \N$.
The function $w$ needs to lie in a convex set $C$, for example, satisfy bound constraints or volume restrictions in
topology optimization, and additionally satisfies the nonconvex pointwise integrality restriction $w(x) \in \Z$ in $\Omega$.
The control function $w$ enters the problem as an input of the state equation (PDE) $Au + N(u,w) = 0$, which we assume to
have a unique solution $u$ for a given $w$. The operator $A$ could be an elliptic differential
operator that includes appropriate boundary conditions and $N(u,w)$ a term with a lower differentiability index on $u$
than $Au$ that may involve a nonlinearity.  The pair of $w$ and the implied solution $u$ to the state equation will minimize
an objective term $j$, which is often a quadratic fidelity term on $u$. Moreover, desired
structures on $w$ can be promoted by means of the regularization term $\alpha R(w)$
for some $\alpha > 0$. A formal setting is provided in \cref{sec:ocp}.

A beneficial choice for $R$ in the context of mixed-integer PDE-constrained optimization
has been the regularization $R = \TV$, where $\TV$ denotes
the total variation seminorm on the space of integrable functions \cite{ambrosio2000functions}.
Such a regularization may be necessary to guarantee the existence of solutions  \cite{leyffer2022sequential,manns2023integer}.
The total variation of an integer-valued function is the $(d-1)$-dimensional volume of the interfaces between its
level sets weighted by the jump heights over the respective interfaces. This choice implies that instances of
\eqref{eq:mipdeco} admit solutions under mild assumptions; see, for example, \cite{burger2012exact,manns2023on}.
Recent work on solving such problems to the satisfaction of the stationarity concept has been studied in
\cite{leyffer2022sequential,manns2023on,severitt2023efficient,manns2023homotopy}.
The computation of lower bounds for a linear PDE, thereby yielding a convex set after omitting the
integrality constraints, and their integration into a branch-and-bound algorithm have recently been
studied by Buchheim et al.\ in
\cite{buchheim2022bounded,buchheim2024compact,buchheim2024parabolicI,buchheim2024parabolicII,buchheim2024parabolicIII}
for $d = 1$.

In this work we strive to obtain convex relaxations and thus lower bounds on \eqref{eq:mipdeco}
in the presence of a structured nonlinearity of the feasible set induced by a specific choice of $N$.
To be precise, we restrict our work to the case that $A$ is an elliptic operator, $R = \TV$, and  the
nonlinearity $N$ in the state equation is in fact a bilinearity, specifically
\[ N(u,w)(x) \coloneqq u(x)w(x) \text{ for a.e.\ } x \in \Omega. \]
We replace the nonlinear state equation by a linear state equation and linear inequalities based on
McCormick envelopes that have been introduced for finite-dimensional nonconvex optimization problems in
\cite{mccormick1976computability}. Our aim is to obtain valid and (ideally) tight lower bounds for
instances of \eqref{eq:mipdeco}. McCormick envelopes and relaxations have already been studied in the
context of global optimization with ODEs finite-dimensional control inputs in  \cite{papamichail2005proof,singer2006global,scott2013improved,scott2011generalized,scott2013nonlinear,sahlodin2011convex,sahlodin2011discretize,wilhelm2019global,ye2024modification}. These articles develop and analyze arguments to show the existence
and computation of solutions to ODEs whose solution trajectories are convex and concave lower and upper bounds on
the solution trajectories of an original ODE that is influenced by means of a finite-dimensional parameter.
In the same spirit such pointwise bounds on the resulting trajectory are transferred to a class of parabolic PDEs
in \cite{azunre2017bounding}. In the context of infinite-dimensional control inputs for
ODE-constrained optimization, the use of McCormick envelopes was proposed in \cite{houska2014branch} and
\cite{scott2011convex}. The latter generalizes ideas from the aforementioned literature to
functions varying pointwise in time. The authors obtain the existence of convex and concave trajectories
that over- and underestimate the trajectories for the parameterized ODE when over- and underestimations
exist pointwise a.e.

We consider a setting, where the existence of such envelopes is straightforward
and provide a possibility to on the one hand only have an approximate relaxation on the original problem
but which on the other hand allows for a way to accelerate an optimization-based
bound-tightening procedure in order to improve the
approximate relaxations.
Generally, such a bound-tightening procedure can become very compute-intensive,
it is called \emph{one of the most expensive bound tightening procedures} in
\cite{gleixner2017three}. Moreover, state-constrained optimal control problems,
in particular with infinitely many state constraints, are notoriously hard to analyze and solve.

Therefore, we introduce an approximation scheme that allows us to work with a finite number of
linear inequalities and variables that is substantially smaller than the number of variables used
for the discretization of the PDE. To this end, we partition the computational domain into
finitely many grid cells and replace the product $uw$ by a product of local averages
$(P_h u)(P_h w)$, where $P_h$ denotes the local averaging defined in \eqref{eq:proj_dg0} below.
This enables us to derive approximate McCormick relaxations that can be described by finitely many linear inequalities.
Using beneficial regularity or continuity properties of the solution to the underlying PDE, specifically
that its solution attains relatively close values in grid cells that are spatially close to each other,
we expect that a moderate mesh size may already yield a good approximation of the lower bound. 
In this way we strive for a much faster computation
in order to obtain approximate lower bounds on \eqref{eq:mipdeco}, which are required for our
goal of solving instances of \eqref{eq:mipdeco} to global optimality. We highlight
that the computation of such lower bounds is also interesting for global optimization in its
own right even without the context of integer-valued input functions.

We believe that our restrictions on the setting are sensible. First, $R = \TV$ is a  useful
regularization for MIPDECOs  because it results in crisp  designs (level sets of the input function).
Moreover, the analytical properties of the $\TV$-regularization term also help  streamline
our analysis although the general strategy can also be carried out for many other choices of $R$.
Similarly, higher-order multilinear terms for $N(u,w)$ involving $u$, for example, $u^k w$, $k \ge 2$,
would make the PDE analysis, namely, the existence of solutions and their regularity, considerably
more complicated. In this regard many of our arguments are streamlined and do not distract from the
analysis and approximation of the McCormick relaxation. Consequently, this article will be read as a recipe to first obtain a pointwise generalization of underestimators for nonconvexities
as are known from finite-dimensional nonconvex optimization and then approximate this pointwise generalization
with finitely many linear inequalities using approximation properties that are due to the PDE setting.

\subsection*{Contribution}
We provide a formal definition of an optimal control problem with the aforementioned features
and derive McCormick envelopes using pointwise constraints. We then introduce a grid that
discretizes the computational domain $\Omega$, and we approximate the inequalities defining
the McCormick envelope by a local averaging over the grid cells. In a second step, we 
discretize the control input function on the same grid, thereby reducing the complexity of
the problem further. We prove the existence of solutions to the approximate convex relaxations
and an estimate on the lower bound for the optimal control problem depending on the mesh size
of the grid. Moreover, we prove that minimizers of the approximate convex relaxations converge
to minimizers of the convex relaxation obtained by imposing the McCormick envelopes by means
of pointwise constraints when driving the mesh size to zero.

We introduce and analyze the aforementioned optimization-based bound-tightening procedure that
allows us to tighten given bounds on the state variable $u$ and in turn to increase the objective
value of the (approximate) McCormick relaxation and thus to tighten the induced lower bounds
on the optimal objective value of the optimal control problem.

We then verify all assumptions imposed on our analysis for an exemplary optimal control problem
that is governed by an elliptic PDE that is defined on $\Omega \subset \R$.
In particular, this example allows us to provide tight values for
all constants in the estimates we impose. We then give details on how we discretize the PDE
and set up a computational experiment where we insert the aforementioned constants. We compute
and compare the approximate lower bounds with and without the bound-tightening procedure applied
with each other to upper bounds with and without an integrality restriction on the
controls in order to assess the quality of the analyzed methodology.
We also provide computational results for a second example that is defined on a two-dimensional domain.
The observations confirm those of the one-dimensional example.

\subsection*{Structure of the remainder}
We continue with a brief introduction to our notation. In
\cref{sec:mccormick} we introduce McCormick envelopes for optimal control problems with 
semilinear state equations that feature a bilinear term. The McCormick envelopes
are first introduced and analyzed in a pointwise fashion. Then a local averaging is introduced
in to approximate the pointwise inequalities by finitely many linear inequalities.
Then employ a piecewise constant control function ansatz, as is desired in
our motivating application of integer optimal control, to obtain a further approximation
with the state variable $u$ still remaining in function space. In \cref{sec:elliptic_example} we
introduce an elliptic PDE that we use as a showcase to verify all of the assumptions that
are imposed for the well-definedness and approximation properties of the different McCormick
settings from \cref{sec:mccormick}. We describe the finite-element discretization
of the state equation for our guiding example, implement the bound-tightening procedure, and
perform our computational experiments in \cref{sec:computational_experiments} (1D)
and \cref{sec:computational_experiments_2D} (2D).
We draw a conclusion in \cref{sec:conclusion}.

\subsection*{Notation}
Let $\Omega \subset \R^d$, $d \in \N$, be a bounded domain. The projection in the space of integrable functions
$L^1(\Omega)$ to piecewise constant functions on a partition $\calQ_h = \{Q_h^1,\ldots,Q_h^{N_h}\}$ of $\Omega$
is denoted by $P_h$; specifically
\begin{gather}\label{eq:proj_dg0}
P_h : L^1(\Omega) \ni f 
\mapsto
\sum_{i=1}^{N_h} \frac{1}{|Q^i_h|}\int_{Q^i_h}f(x) \dd x \chi_{Q^i_h} \in L^1(\Omega).
\end{gather}
For a measurable set $A \subset \R^d$, $d\in \N$, we denote its $d$-dimensional Lebesgue measure
by $|A|$. We say that a function $w \in L^1(\Omega)$ is of bounded variation if
its total variation seminorm $\TV(w)$ is finite. The Banach space of functions in $L^1(\Omega)$
of bounded variation with norm $\|\cdot\|_{\BV} = \|\cdot\|_{L^1} + \TV(\cdot)$
is then denoted by $\BV(\Omega)$; see \cite{ambrosio2000functions} for details on functions
of bounded variation. We recall that the total variation seminorm for $w \in L^1(\Omega)$
is defined as
\[ \TV(w) \coloneqq \sup\left\{
\int_{\Omega} w(x) \dvg \phi(x) \dd x \,\middle|\,
\phi \in C_c^{1}(\Omega,\R^d),\ \max_{x \in \Omega} \|\phi(x)\| \le 1 \right\},
\]
where $\|\cdot\|$ denotes the Euclidean norm on $\R^d$.
On one-dimensional domains $\Omega = (a,b)$ for $a$, $b \in \R$, every element of $\BV(\Omega)$
has a left-continuous representative, and the total variation is the sum of the jump heights
of the left-continuous representative (this interpretation also works for the right-continuous
representative or other so-called good representatives; see \cite{ambrosio2000functions}).

\section{McCormick envelopes for optimal control problems
with state equations that have bilinear terms}\label{sec:mccormick}
In this section we derive and analyze McCormick envelopes and relaxations for a class of
nonconvex optimal control problems where the nonconvexity
stems from bilinear terms in the state equation.
We first introduce a prototypical optimal control problem in \cref{sec:ocp}.
We then derive pointwise a.e.\ McCormick envelopes in \cref{sec:mccormick_pointwise}
and, subsequently, provide a local averaging in two stages that yields approximate lower bounds but implies  problems,
still in function space, which feature only finitely many additional
linear inequalities in \cref{sec:mccormick_h}. In \cref{sec:obbt} we prove the validity
of a bound-tightening procedure that may improve (increase) the (approximate) lower bounds.

\subsection{Optimal control problem
with nonconvex
bilinearity}\label{sec:ocp}
We consider the following prototypical setting with
a nonconvexity induced by
a bilinear term in the state equation,
 but note that the strategy we develop can be 
transferred to many other settings with lower or slightly different regularity assumptions
with a small amount of modifications.

Let $\Omega$ be a bounded domain. Let $U$ be a reflexive Banach space that is compactly embeded into 
$H \coloneqq L^2(\Omega)$, that is, $U \hookrightarrow^c H$,
so that we may work with the triple
$U \hookrightarrow^c H \cong H^* \hookrightarrow^c U^*$ of compact embeddings
and the isometric isomorphic identification $H \cong H^*$. In optimal control 
terminology, $U$ is the \emph{state space}. For the \emph{control space}, we consider the space
$W \coloneqq \BV(\Omega)$ of functions of bounded variation, which is motivated by the 
compactness and approximation properties it  provides in the control space and our intended 
application in the context of integer optimal control. Specifically, 
$\BV(\Omega)$-regularity allows us to approximate a 
function by its average on a partition of the domain with an error that is bounded by a 
scalar multiple of the mesh  size. Note that other spaces like the Sobolev space $H^1(\Omega)$
 also provide this property.

The optimal control problem is
\begin{gather}\label{eq:ocp}
\begin{aligned}
\min_{u, w}\ & j(u, w) + \alpha \TV(w) \\
\text{ s.t.\ } & A u + u w =  f \enskip\text{ in } U^*,  \\
& w \in C, u \in U
\end{aligned}\tag{OCP}
\end{gather}
for a nonempty, bounded, closed, and convex set $C \subset H$, an objective $j : U \times H \to \R$
that is Lipschitz continuous on bounded sets and a linear and bounded differential
operator $A : U \to U^*$ with bounded inverse $A^{-1} : U^* \to U$.
The nonconvexity of the
problem stems from the
bilinear term $uw$ in
the state equation.

Note that the assumptions on $C$ imply that it is a weakly sequentially compact subset of $H$
so that we obtain the regularity $w \in H$ for all feasible $w$. Together
with the $\TV$-seminorm in the objective, we obtain the regularity $w \in W$ for all feasible $w$
with finite objective value and the assumed structure $j : U \times H \to \R$ is well-defined.
Moreover, we note that the continuous embedding $W \hookrightarrow H$ holds for $d \in \{1,2\}$ 
so that we can equivalently consider $W$ or $W \cap H$ as control space. For $d \ge 3$, defining
$W \cap H$ as control space here is slightly more specific but not necessary for our analysis
and would complicate our notation later.
The vector $f \in U^*$ is a fixed datum that parameterizes the optimization problem, and
we assume that the product $uw$ is a measurable function that is (also) an element of $U^*$
for all tuples $(u,w) \in U \times W$. This follows, for example, if $C$ is bounded
in $L^\infty(\Omega)$, too. The term $\TV(w)$ denotes the total variation seminorm
of $w$ that enforces its required $\BV(\Omega)$-regularity, and $\alpha > 0$ is a
positive scalar. Assuming that the state equation and the original optimal
control problem admit unique solutions, the presence of the term $uw$
implies a nonlinear dependence of the solution to the state equation
on $w$, which in turn implies that \eqref{eq:ocp} is generally nonconvex even if $j$
is convex. 

\subsection{Pointwise McCormick envelopes}\label{sec:mccormick_pointwise}
McCormick envelopes have been introduced for finite-dimensional
nonconvex optimization problems in \cite{mccormick1976computability}.
The key idea for the case of \eqref{eq:ocp} is to derive a new optimal 
control problem, the so-called McCormick envelope, by relaxing 
the state constraints. Specifically, the bilinear term is replaced
by a new variable $z$, which is a measurable function
that is also an element of $U^*$, and a set of linear inequalities
on $z$, $u$, and $w$ that preserve the feasibility of the state
equation; that is, if $u$ and $w$ satisfy $Au = uw + f$, then
the choice $z = uw$ satisfies all of the newly introduced
linear inequalities. Here, we make the following assumption.
\begin{assumption}\label{ass:mcc}
Let there be bounds $u_\ell$, $u_u \in H$ and $w_\ell$, $w_u \in L^\infty(\Omega)$
so that $u_\ell \le u \le u_u$ and $w_\ell \le w \le w_u$ hold pointwise a.e.\
for all $(u, w) \in U \times W$ that solve $Au + uw = f$ and satisfy $w \in C$.
\end{assumption}
Clearly, the multiplication of these bounds, $u_\ell w_\ell$, $u_\ell w_u$, $u_u w_\ell$, $u_u w_u \in H$ by H\"older's
inequality. We note that the assumption of bounds on the state is quite  strong, but we emphasize that such bounds are generally
required only in the subset of the computational domain $\Omega$ on which $N$ actually acts.
For example, terms like $N(u, w) = \chi_{A} u w$ for compact subsets $A \subset \subset \Omega$ may occur
in topology optimization problems like in \cite{haslinger2015topology} because the control design
may be restricted to $A$ and is extended by zero to the complement of $A$. Consequently, higher interior regularity
of the PDE solutions can help to establish implementable (i.e., generally uniform) bounds.
For rich classes of PDEs that are governed by an elliptic operator like in our setting, $L^\infty(\Omega)$-bounds
on $u$ can be established by following of the strategy and results in Appendix B in \cite{kinderlehrer2000introduction} under 
assumptions on $f$ and the coefficients of the elliptic operator, see also Theorem 4.5 and its proof in 
\cite{troltzsch2010optimal}.

We note that the different parts of the proofs do not require the full assumed regularity
here, but again we prefer this slightly more restrictive setting in the interest of
a cleaner presentation. Under \cref{ass:mcc}, the McCormick relaxation of 
\eqref{eq:ocp} can be derived pointwise a.e.\ in a similar way as 
for finite-dimensional problems; see \cite{mccormick1976computability}:
\begin{gather}\label{eq:mcc}
\begin{aligned}
\min_{u, w, z}\ & j(u, w) + \alpha \TV(w) \\
\text{ s.t.\ }\,
& A u + z = f \enskip\text{ in } U^*,  \\
&z \ge u_\ell w + u w_\ell - u_\ell w_\ell\enskip\text{a.e.},\\ 
&z \ge u_u w + u w_u - u_u w_u\enskip\text{a.e.},\\
&z \le u_u w + u w_\ell - u_u w_\ell\enskip\text{a.e.},\\
&z \le u_\ell w + u w_u - u_\ell w_u\enskip\text{a.e.},\\
&u \in U,\; u_\ell \le u \le u_u\enskip\text{a.e.},\\
&w \in C,\; w_\ell \le w \le w_u\enskip\text{a.e.}
\end{aligned}\tag{McC}
\end{gather}
We obtain that \eqref{eq:mcc} is a relaxation of \eqref{eq:ocp} with a convex feasible set in the proposition below.
As mentioned in the introduction, (generalized \cite{scott2011generalized}) McCormick relaxations have been
studied for ODEs with finite-dimensional inputs to obtain concave and convex upper and lower bounds
for their solution trajectories
\cite{papamichail2005proof,singer2006global,scott2013improved,scott2011generalized,scott2013nonlinear,sahlodin2011convex,sahlodin2011discretize,wilhelm2019global,ye2024modification}. This implies analogous relaxation results for their settings.
In the context of ODE-constrained optimal control, such a result was shown for a pointwise (that is infinite-dimensional)
control for nonlinearities given by factorable functions and appropriate Lipschitz assumptions,
see Lemma 1, Assumption 3, and Theorem 1 in \cite{scott2011convex}. Their observations can likely be used to transfer
some of our ideas to more general nonlinearities than present in \eqref{eq:ocp}. We stress that the analysis of the PDE
becomes increasingly difficult when higher-degree monomials are present and generally requires a study of the specific case
to determine if the PDE has a solution and if yes if it lies in a function space that provides enough regularity for the analysis below.
Using McCormick relaxations for such infinite-dimensional control settings and envelopes around solution trajectories
was also proposed in \cite{houska2014branch}.
\begin{proposition}\label{prp:ocp_mcc_basics}
Let \cref{ass:mcc} hold.
\begin{enumerate}
\item If $(u,w)$ is feasible for \eqref{eq:ocp}, then $(u,w,z)$
with $z = uw$ is feasible for \eqref{eq:mcc}. In particular,
the infimum \eqref{eq:mcc} is a lower bound on the infimum
of \eqref{eq:ocp}.
\item The feasible set of \eqref{eq:mcc} is convex and bounded
in $U \times H \times H$.
\end{enumerate}
\end{proposition}
\begin{proof}
Regarding feasibility, the state equation, the inclusion in $C$,
and the two last pointwise inequalities are immediate.
The satisfaction of the four additional pointwise inequalities
for the choice $z = uw$ follows as in the finite-dimensional case by
rearranging them. For example, we obtain
$(u - u_\ell)w \ge (u - u_\ell) w_\ell$ a.e.\ for the first inequality.
The state equation in \eqref{eq:mcc} and the additional inequalities 
are affine in $(u,w,z)$. Moreover, $C$ is assumed to be convex
so that \eqref{eq:mcc} has a convex feasible set.
The feasible $w$ are bounded in $H$. In combination with the pointwise bounds on $u$, 
we obtain boundedness of the feasible $z$ in $H$ from the McCormick inequalities.
In turn, the continuous invertibility of $A$ and the embedding $H \cong H^* \hookrightarrow^c U^*$
imply the boundedness of the feasible $u$ in $U$.
\end{proof}
\begin{remark}\label{rem:mcc_bound_constraint}
Because PDE-constrained optimization problems generally feature a large number of variables after
discretization due to their distributed nature, the pointwise McCormick inequalities add many linear
constraints---on the order of six times the number of discretization cells---to a later
fully discretized problem if one enforces them on every degree of freedom of a finite-element
discretization of the state variable.

Without any additional constraints, we do not know whether we can actually bound $z$ uniformly in $H$, however, so we add these bound constraints on $u$, which are again satisfied by all feasible
points for \eqref{eq:ocp}.
\end{remark}
\begin{remark}
For the product $xy$ of a binary-valued variable $x$ and a fractional-valued
variable $y$ with $y_\ell \le y \le y_u$, as is present in our guiding MIPDECO example,
a envelopes arises by replacing the product with a new variable $z$ and use the so-called
big-M linearization, where $y_\ell$ and $y_u$ assume the roles of the big-Ms. This yields
the linear inequality system
\begin{align*}
y_\ell x \le z \le y_u z\quad\text{and}\quad (x - 1)y_u \le z - y \le (x - 1)y_\ell.
\end{align*}
Therein, the two bounds on $z$ correspond to the first and third inequality in the McCormick envelopes 
as introduced in \eqref{eq:mcc} and the bounds on $z - y$ correspond to the second and fourth
inequality in the McCormick envelopes so that these two approaches are equivalent for our guiding
example.
\end{remark}

\begin{proposition}\label{prp:existence}
In the setting of \cref{sec:ocp}, the problems \eqref{eq:ocp} and \eqref{eq:mcc} admit
solutions.
\end{proposition}
\begin{proof}
The existence of solutions for \eqref{eq:ocp} follow with standard
arguments from the direct method of calculus of variations.

By construction of \eqref{eq:mcc} as a relaxation of \eqref{eq:ocp}, its
feasible set is nonempty, and there exists a minimizing sequence
$(u^n, w^n, z^n)_n \subset U \times H \times H$ of feasible points, which has at least one accumulation point 
$(u,w,z) \in H \times H \times H$ with respect to
weak convergence in $H \times H \times H$, which follows directly from
\cref{prp:ocp_mcc_basics}. Because $A^{-1}$ is a bounded, linear operator,
$u$ satisfies the state equation and is an element of $U$, too.
Moreover, because $H$ is compactly embedded in $U^*$ and thus
$z^n \to z$ holds in $U^*$, we even obtain
\[ u^n = A^{-1}(f - z^n) \to A^{-1}(f - z) = u \text{ in } U \]
and in turn $u^n \to u$ in $H$ and pointwise a.e.\ after restricting to a suitable subsubsequence.
The pointwise a.e.\ convergence yields $u_\ell \le u \le u_u$.

The assumptions on $C$ imply $w \in C$. Because $j$ is bounded below, $(\TV(w^n))_n$ is
bounded; and  again after restricting to a suitable subsubsequence,  we obtain $w^n \to w$
in $L^1(\Omega)$ and pointwise a.e.\ from the weak-$^*$ sequential compactness properties of the
$\TV$-seminorm; see Theorem 3.23 in \cite{ambrosio2000functions}. We obtain $w_\ell \le w \le w_u$.

It remains to assert that the weak limit $z$ satisfies the McCormick inequalities before
proving the optimality. This follows from \cref{lem:lower_bound_for_weak_limit}.

The continuity properties of $j$ and the lower semi-continuity of the
$\TV$-seminorm with respect to convergence in $L^1(\Omega)$ and in turn also in $H$
yield that $(u,w,z)$ is a minimizer.
\end{proof}
\begin{remark}
In our examples, the space $W$ will be (a subspace of) the space of functions of bounded 
variation $\BV(0,1)$, which is not reflexive. There are important sets $C$ that are subsets
of such spaces and also convex, bounded, and weakly sequentially closed in $H$.
An example is $\{w \in \BV(0,1) \,|\, \TV(w) \le \kappa\}$ for some $\kappa > 0$; see
also \cite{buchheim2024parabolicI} for such a constraint set in the context of integer
optimal control.
\end{remark}

The challenges of the pointwise McCormick envelopes for \eqref{eq:mcc} have already
been sketched in \cref{rem:mcc_bound_constraint}. The situation is further complicated if
one wants to improve the lower bound induced by \eqref{eq:mcc}. A straightforward application
of a bound-tightening procedure implies that one alternatingly picks $\tilde{x} \in \Omega$
and solves (one of) the optimization problems
\[ \inf_{u,w,z} u(\tilde{x}) \text{ s.t.\ } (u,w,z) \text{ is feasible for } \eqref{eq:mcc} \]
and
\[ \sup_{u,w,z} u(\tilde{x}) \text{ s.t.\ } (u,w,z) \text{ is feasible for } \eqref{eq:mcc} \]
in order to then update $u_\ell(\tilde{x})$ and $u_u(\tilde{x})$ to the computed values. Of course,
this is  sensible only if $u$  in $C(\bar{\Omega})$ or  has a meaningful pointwise interpretation.
In a practical implementation, such optimization problems may, for example, be solved for nodes 
$\tilde{x}$ of a finite-element discretization with a nodal basis.
Since improving the bounds yields further possible improvements, it makes sense to repeatedly solve
such problems until no further progress is made. In total, the bound-tightening procedure can become
computationally very expensive. Therefore, in the next section we analyze a local averaging of the
nonlinearity that leads to approximate lower bounds by means of a grid whose mesh size can then be
chosen coarser than the mesh size of the grid that is used to the discretize the state variable.

\subsection{Approximate McCormick relaxations
by local averaging}\label{sec:mccormick_h}
We propose, analyze, and assess a local averaging that modifies the
problem before introducing
the variable $z$. Thus, the variable $z$ and in a second approximation step $w$, which
only enters the relaxed PDE implicitly through $z$, and the bounds on the product $uw$ are
approximated by means of a coarse grid that allows us to reduce the size of the resulting
optimization problem. We give an overview on the approximation arguments of this subsection in
\cref{fig:approximation_arguments}.
\begin{figure}
\centering
\begin{tikzpicture}
\filldraw[black!10!white]  (2,0) rectangle (6,1);
\node[align=left] at (4,0.5) {\eqref{eq:ocp}};
\draw[->,very thick] (4, 0) -- (4, -1);
\node[align=left,anchor=west] at (4,-0.5)
{{\footnotesize lower bound} \\ 
	{\footnotesize (\cref{prp:ocp_mcc_basics})}};
\filldraw[black!10!white]  (2,-2) rectangle (6,-1);
\node[align=left] at (4,-1.5)
{\eqref{eq:mcc}};
\draw[->,very thick] (4, -2) -- (4, -3);
\node[align=left,anchor=west] at (4,-2.5)
{{\footnotesize approximation of $z$} \\ {\footnotesize fewer inequality constraints}};
\filldraw[black!10!white]  (2,-4) rectangle (6,-3);
\node[align=left] at (4,-3.5)
{\eqref{eq:mcch}};
\draw[->,very thick] (4, -4) -- (4, -5);
\node[align=left,anchor=west] at (4,-4.5)
{{\footnotesize approximation of $w$ and $\TV$} \\ {\footnotesize fewer variables}};
\filldraw[black!10!white]  (2,-6) rectangle (6,-5);
\node[align=left] at (4,-5.5)
{\eqref{eq:mcchh}};
\draw[very thick,->,bend left] (6,.5) to (9.5,-1.5) to (6,-3.5);
\node[align=left,anchor=west] 
at (9.5,-1) {{\footnotesize approx.\ lower bound} \\ {\footnotesize (\cref{thm:mcch_approx_lb_on_ocp})}};
\draw[very thick,->,bend left] (6,.5) to (9.5,-2.5) to (6,-5.5);
\node[align=left,anchor=west] 
at (9.5,-3) {{\footnotesize approx.\ lower bound} \\ {\footnotesize (\cref{thm:mcchh_approx_lb_on_ocp})}};
\draw[very thick,->,bend left] (2,-5.5) to (2,-1.5);
\node[align=left,anchor=east] 
at (1.25,-3.5) {{\footnotesize $\Gamma$-convergence} \\ {\footnotesize (\cref{prp:gamma_mcchh_to_mcc})}};
\end{tikzpicture}
\caption{Two-stage approximation of \eqref{eq:mcc} by reduced problems \eqref{eq:mcch}
and \eqref{eq:mcchh} in order to obtain approximate lower bounds
on \eqref{eq:ocp}.}\label{fig:approximation_arguments}
\end{figure}
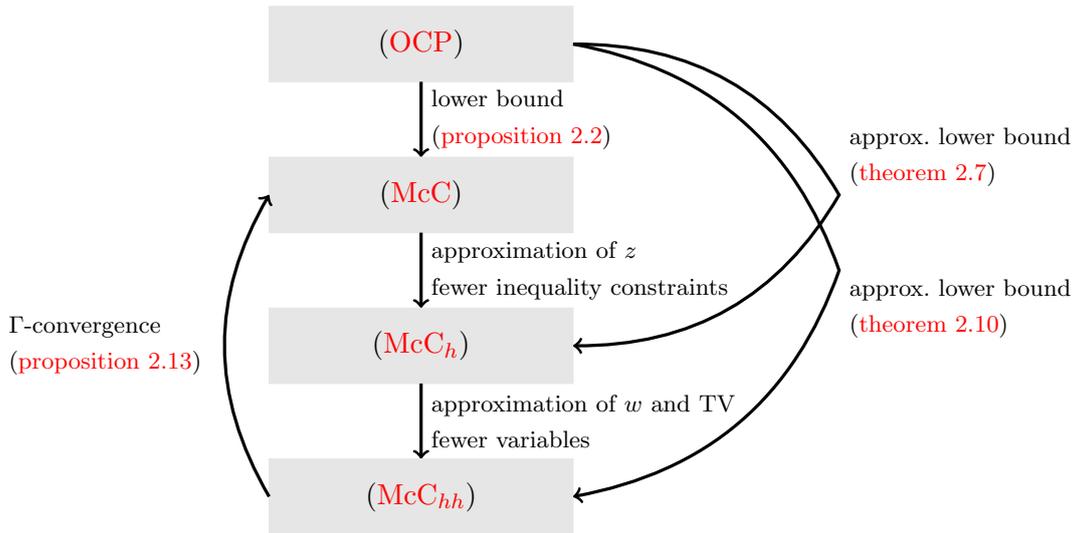

Let $\calQ_h = \{Q_h^1,\ldots,Q_h^N\}$ be a partition of the domain $\Omega$ into
intervals / squares / cubes / ... with mesh size $h$. Locally averaging the nonlinearity
of the state equation from \eqref{eq:mcc} means that $Au + uw = f$ is replaced by the variant
\begin{gather}\label{eq:pdeh}
A u + (P_h u)(P_h w) = f
\end{gather}
with $P_h$ being defined in \eqref{eq:proj_dg0}. We define the locally averaged McCormick relaxation
with respect to the partition $\calQ_h$ as
\begin{gather}\label{eq:mcch}
\begin{aligned}
\min_{\substack{u, w,\\ z_1,\ldots,z_{N_h}}}\ 
& j(u, w) + \alpha\TV(w)\\
\text{ s.t.\ }\ \ 
& A u + z = f \enskip\text{ in } U^*,  \\
&z = \sum_{i=1}^{N_h} z_i \chi_{Q^i_h},\\
&z_i \ge u_\ell^i (P_h w) + (P_h u) w_\ell^i - u_\ell^i w_\ell^i
\enskip & \text{ on }Q^i_h\text{ for all } i \in \{1,\ldots,N_h\},\\ 
&z_i \ge u_u^i (P_h w) + (P_h u) w_u^i - u_u^i w_u^i
\enskip & \text{ on }Q^i_h\text{ for all } i \in \{1,\ldots,N_h\},\\ 
&z_i \le u_u^i (P_h w) + (P_h u) w_\ell^i - u_u^i w_\ell^i
\enskip & \text{ on }Q^i_h\text{ for all } i \in \{1,\ldots,N_h\},\\ 
&z_i \le u_\ell^i (P_h w) + (P_h u) w_u^i - u_\ell^i w_u^i
\enskip & \text{ on }Q^i_h\text{ for all } i \in \{1,\ldots,N_h\},\\ 
&u \in U, u_\ell^i \le P_h u \le u_u^i
\enskip & \text{ on }Q^i_h\text{ for all } i \in \{1,\ldots,N_h\},\\
&w \in C,\; w_\ell \le w \le w_u\enskip\text{a.e.},\\
&z_1,\ldots,z_{N_h} \in \R.
\end{aligned}\tag{McC$_h$}
\end{gather}
Note that \eqref{eq:mcch} is still in function space regarding $u$ and $w$.
In \eqref{eq:mcch} the statement that the four McCormick inequalities on $z_i$  hold on $Q_h^i$ 
might be misunderstood. We therefore highlight that $P_h u$ and $P_h w$ are constant on the cells 
$Q_h^i$ so that there are exactly four McCormick inequalities per cell $Q_i^h \in \calQ_h$.
The same argument applies for the bounds on $P_h u$. The bounds $u_\ell^i \in \R$ and $u_u^i \in \R$ 
assume the roles of the functions $u_\ell$ and $u_u$ per grid cell for the locally averaged function 
$P_h u$. Specifically, the function $P_h u$
is bounded from below by $u_\ell^i$ and from above by $u_u^i$ on the $i$-th grid cell
$Q^i_h$.
Similarly, the bounds $w_\ell^i$ and $w_u^i$ assume the roles of $w_\ell$ and $w_u$.

By employing only finitely many bounds on $P_h u$, we can circumvent the aforementioned regularity
problems if the number of constraints is not too large; that is, its Lagrange multiplier may 
sensibly be interpreted as a vector in $\R^n$. Moreover, the local averaging opens up
possibilities for adaptive
McCormick relaxations when using a corresponding adaptive refinement of the partition $\calQ_h$
and will also reduce the computational effort for improving the (approximate) lower bound. We can now
impose inequalities on the locally averaged state $P_h u$ and tighten them algorithmically in order to
reduce the gap between the relaxation and the original problem. We now impose an assumption that implies
that \eqref{eq:mcch} admits a solution and is an approximate relaxation of \eqref{eq:ocp}.
This means that the optimal objective of \eqref{eq:ocp} is bounded from below by the
objective of \eqref{eq:mcch} minus a bound on the approximation error.

\begin{assumption}\label{ass:mcch}
In addition to the setting from \cref{sec:ocp}, we assume the following.
\begin{enumerate}
\item\label{itm:partitions}
Let $\{\calQ_h\}_h$ be a sequence of partitions of the domain $\Omega$ with mesh sizes $h$.
\item\label{itm:phwphu_pdesolexists}
Let the equation \eqref{eq:pdeh} have a unique solution for all $w \in \BV(\Omega)$ with $w(x) \in [w_\ell,w_u]$ a.e.
\item\label{itm:richness_of_feasible_set} For all $h$ and all $Q_i \in \calQ_h$,
let the bounds $u_\ell^i$, $u_u^i$, $w_\ell^i$, $w_u^i \in \R$ be such that
$u_\ell^i \le (P_h u)|_{Q_i} \le u_u^{i}$
and $w_\ell^i \le (P_h w)|_{Q_i} \le w_u^i$ hold for all $w \in C \cap [w_\ell, w_u]$ and $u \in U$ that solve
$Au + (P_h u) (P_h w) = f$.
\end{enumerate}	
\end{assumption}
We will verify all parts of \cref{ass:mcch} for an example in 
\cref{sec:elliptic_example}.
\begin{theorem}\label{thm:mcch_approx_lb_on_ocp}
Let \cref{ass:mcch} hold. Then \eqref{eq:mcch} has a convex feasible set that is bounded in
$U \times H \times \R^{N_h}$ and admits a minimizer.
Let $(u, w)$ be feasible for \eqref{eq:ocp}. Then the tuple $(u_h,w_h,z_h)$  with
$w_h \coloneqq w$, $u_h$ being the unique solution to \eqref{eq:pdeh}, and
$z_h \coloneqq (P_h u_h) (P_h w_h)$ is feasible for \eqref{eq:mcch}, and the objective
value satisfies
\begin{gather}\label{eq:uhwh_uw_h_approx}
|j(u_h, w_h) + \alpha \TV(w_h) - j(u, w) - \alpha \TV(w)|
\le L_u \|u - u_h\|_U,
\end{gather}
where $L_u$ is the Lipschitz constant of $j$ with respect to its first argument on the feasible set of
\eqref{eq:mcch}. Moreover,
\begin{gather}\label{eq:mcch_approx_lb_ocp}
m_{\eqref{eq:ocp}} \ge m_{\eqref{eq:mcch}} - L_u \|\bar{u} - \bar{u}_h\|_U
\end{gather}
if $m_{\eqref{eq:mcch}}$ denotes the infimum of \eqref{eq:mcch},
$m_{\eqref{eq:ocp}}$ denotes the infimum of \eqref{eq:ocp},
$(\bar{u}, \bar{w})$ denotes the minimizer of \eqref{eq:ocp},
and $\bar{u}_h$ the corresponding solution to \eqref{eq:pdeh}.
\end{theorem}
\begin{proof}
The convexity and boundedness of the feasible set follow as in \cref{prp:ocp_mcc_basics}.
The feasibility of  $(u_h,w_h,z_h)$ for \eqref{eq:mcch} follows by construction.
Because the feasible set of \eqref{eq:mcc} is nonempty (see \cref{prp:existence}),
there exist feasible points of \eqref{eq:mcch}.
Because the feasible set of \eqref{eq:mcch} is bounded, the desired constant $L_u$ exists
by assumption on $j$.

Consequently, we can consider a minimizing sequence $\{(u_h^k,w_h^k,z_h^k)\}_k$
for \eqref{eq:mcch} and apply similar arguments as in the proof of 
\cref{prp:existence} in order to prove existence of a minimizer.
The $z_h^k$ have a finite-dimensional piecewise constant ansatz and are bounded because of
the boundedness of $w$ implied by the properties of $C$ so that they
admit a \mbox{(norm-)}convergent subsequence in $H$ with limit $\bar{z}_h$. Because of
the continuous invertibility of $A$, there is a corresponding subsequence of 
$\{u_h^k\}_k$ that converges to $\bar{u}_h = A^{-1}(-\bar{z}_h + f)$. By
possibly passing to a further subsequence, the corresponding subsequence
of $\{w_h^k\}_k$ converges weakly$^*$ to $\bar{w}_h$ in $W$.
This implies $\bar{w}_h \in C$ and $w_\ell \le \bar{w}_h \le w_u$.
Passing to a further subsequence, we obtain pointwise a.e.\ convergence for
all three subsequences and in turn feasibility of the limit triple
$(\bar{u}_h, \bar{w}_h, \bar{z}_h)$ for \eqref{eq:mcch}, where we use that the 
projection to the piecewise constant functions is continuous with respect to all 
$L^p$-norms.
For the objective, we obtain that the first term converges because of its assumed
continuity properties. Moreover, the $\TV$-term is lower semicontinuous
with respect to convergence in $H$; and, in turn, we obtain that 
$(\bar{u}_h, \bar{w}_h, \bar{z}_h)$ realizes the infimal objective value and
is thus a minimizer.

The approximation bound \eqref{eq:uhwh_uw_h_approx} follows from the Lipschitz
continuity of $j$ and the assumed choice $(u_h,w_h,z_h)$.

To prove the approximate lower bound \eqref{eq:mcch_approx_lb_ocp},
let $(\bar{u},\bar{w})$ be a minimizer to \eqref{eq:ocp}, and let
$\bar{w}_h = \bar{w}$. Then \cref{ass:mcch} \ref{itm:phwphu_pdesolexists} gives a unique solution
$\bar{u}_h$ to \eqref{eq:pdeh}, and we choose $\bar{z}_h = (P_h u_h)(P_h w)$.
\Cref{ass:mcch} \ref{itm:richness_of_feasible_set} yields that the triple 
$(\bar{u}_h, \bar{w}_h, \bar{z}_h)$ is feasible for \eqref{eq:mcch}.
We deduce by means of $\bar{w}_h = \bar{w}$ and the optimality
of $(\bar{u},\bar{w},\bar{z})$ that
\begin{align*}
m_{\eqref{eq:ocp}} = j(\bar{u},\bar{w}) + \alpha \TV(\bar{w})
&\ge j(\bar{u}_h,\bar{w}_h) + \alpha \TV(\bar{w}_h) - L_u \|u - u_h\|_U\\
&\ge m_{\eqref{eq:mcch}} - L_u \|u - u_h\|_U.
\end{align*}
\end{proof}
We note that approximation results on the difference $\|u - u_h\|_U$ that are required
to obtain a meaningful bound can generally be obtained for broad classes of elliptic and
parabolic PDEs.

The locally averaged McCormick relaxation \eqref{eq:mcch} has the drawback
that while only a local average of $w$ is required for the McCormick inequalities, the full information is required in order to represent the total 
variation term correctly. To reduce the
number of optimization variables further, one is tempted to replace $\TV(w)$ by
$\TV(P_h w)$. Then one could use $P_h w = \sum_{i=1}^{N_h} w_i \chi_{Q^i_h}$ for real
coefficients $w_i$, $i \in \{1,\ldots,N_h\}$, in order to replace the 
optimization over $w$ in $\BV(\Omega)$ by an optimization over $\R^N$.
The lower semicontinuity inequality
$\TV(w) \le \liminf_{h\searrow 0}\TV(P_h w)$ is strict in general if
the dimension of $\Omega$ is larger than one. To make things worse,
even a $\Gamma$-convergence result cannot be obtained directly because the 
geometries of the functions in the limit $h \searrow 0$ are restricted
due to the restrictions of the geometry of the elements of
the partitions $\calQ_h$. This is noted at the end of section 1 in \cite{chambolle2021approximating}
with references to \cite{hochbaum2001efficient} and \cite{chambolle2012parametric} that employ
discrete (anisotropic) approximations of the total variation. Specifically, the isotropic unit ball
that is due to the Euclidean norm in the definition of the total variation is not recovered in
the limit process for $h \searrow 0$, and other limit functionals that have nonisotropic unit balls
may arise for periodic discretizations; this situation is analyzed and reported in \cite{cristinelli2023conditional}.

We introduce a variant of the McCormick relaxation where both the McCormick inequalities and the control discretization employ a local averaging with respect to the same grid based on the following assumption. To be able to obtain the desired approximation bound and a $\Gamma$-convergence 
to the limit problem \eqref{eq:mcc}, we assume that a regularization of
the total variation penalty, such as the $L^2$-regularization of the dual 
formulation presented in \cite{chambolle2017accelerated}, is coupled to
the grid used for the local averaging.
\begin{assumption}\label{ass:mcchh}
In addition to \cref{ass:mcch}, we assume that there is a sequence of approximations $(\TVh)_h$ of $\TV$
that is consistent with local
averaging. Specifically, we assume for all $h$ that
$\TVh : L^1(\Omega) \to [0,\infty]$ is lower semicontinuous,
$\TVh \circ P_h \le \TV$, and $\TVh \circ P_h \le \TVh$.
\end{assumption}
\begin{remark}
In one-dimensional domains, the choice $\TV = \TVh$ immediately satisfies \cref{ass:mcchh}, which
follows from \cref{lem:1DTVPh_gamma} below. On multidimensional domains, one can, for example,
choose suitable finite-element discretizations of the total variation seminorm like a
approximation with lowest-order Raviart--Thomas elements; see \S4.1 in \cite{chambolle2021approximating}.
Note that a correct discretization of the $\TV$
seminorm is not trivial in our intended context of mixed-integer
PDE-constrained and we refer to the recent article
\cite{schiemann2024discretization} for details on the
difficulties and a solution based on the aforementioned
lowest-order Raviart--Thomas elements.
\end{remark}
The variant of the McCormick relaxation that replaces $\TV(w)$ by $\TVh(w)$ is:
\begin{gather}\label{eq:mcchh}
\begin{aligned}
\min_{\substack{u, w_1,\ldots,w_{N_h},\\
		z_1,\ldots,z_{N_h}}}\ & j(u, w) + \alpha\TVh(w)\\
\text{ s.t.\ }\ \ \ \
& A u + z = f \enskip\text{ in } U^*,  \\
& w = \sum_{i=1}^{N_h} w_i \chi_{Q^i_h}, z = \sum_{i=1}^{N_h} z_i \chi_{Q^i_h},\\
&z_i \ge u_\ell^i w_i + (P_h u) w_\ell^i - u_\ell^i w_\ell^i
\enskip & \text{ on }Q^i_h\text{ for all } i \in \{1,\ldots,N_h\},\\ 
&z_i \ge u_u^i w_i + (P_h u) w_u^i - u_u^i w_u^i
\enskip & \text{ on }Q^i_h\text{ for all } i \in \{1,\ldots,N_h\},\\ 
&z_i \le u_u^i w_i + (P_h u) w_\ell^i - u_u^i w_\ell^i
\enskip & \text{ on }Q^i_h\text{ for all } i \in \{1,\ldots,N_h\},\\ 
&z_i \le u_\ell^i w_i + (P_h u) w_u^i - u_\ell^i w_u^i
\enskip & \text{ on }Q^i_h\text{ for all } i \in \{1,\ldots,N_h\},\\ 
&u \in U,\; u_\ell^i \le P_h u \le u_u^i &  \text{ on }Q^i_h\text{ for all } i \in \{1,\ldots,N_h\},\\
&w \in C,\; w_\ell \le w \le w_u \text{ a.e.},\\
&z_1,\ldots,z_{N_h} \in \R
\end{aligned}\tag{McC$_{hh}$}
\end{gather}
Note that \eqref{eq:mcchh} is still in function space regarding $u$.
With the help of \cref{ass:mcchh}, we can show that \eqref{eq:mcchh}
is an approximate lower bound on \eqref{eq:mcch} and in turn on \eqref{eq:ocp}.
We note that the inequalities in \cref{ass:mcchh} are implementable in practice and can be relaxed to approximate
inequalities  with an $h$-dependent error that tends to zero for $h \searrow 0$.
\begin{theorem}\label{thm:mcchh_approx_lb_on_ocp}
Let \cref{ass:mcchh} hold. Then \eqref{eq:mcchh} has a convex feasible set 
that is bounded in $U \times \R^{N_h} \times \R^{N_h}$ and admits a minimizer.
We define $w_h \coloneqq \sum_{i=1}^{N_h} w_i^h \chi_{Q^i_h}$ for
a feasible point $(u_{hh} ,w_h^1,\ldots,w_h^{N_h},z_h^{1},\ldots,z_h^{N_h})$
of \eqref{eq:mcchh}. 
Then the tuple $(u_{hh}, w_h, z_h^1,\ldots,z_h^{N_h})$ is feasible for \eqref{eq:mcch}.

Let $(u_{h}, w_h, z_h^{1},\ldots,z_h^{N_h})$ be feasible for \eqref{eq:mcch}. Then
$(u_{h}, w_h^1,\ldots,w_h^{N_h},z_h^{1},\ldots,z_h^{N_h})$
with $w_h^i \coloneqq (P_h w_h)|_{Q_h^i}$,
$i \in \{1,\ldots,N_h\}$, is feasible for \eqref{eq:mcchh}. Moreover,
\begin{gather}\label{eq:mcchh_approx_lb} 
j(\bar{u}_h, w_h) + \alpha \TV(w_h) \ge m_{\eqref{eq:mcchh}} 
- L_w \|w_u - w_\ell\|_{L^\infty}^{\frac{1}{2}} \sqrt{d}^\frac{1}{2} \TV(\bar{w}_h)^\frac{1}{2} h^\frac{1}{2},
\end{gather}
where $m_{\eqref{eq:mcchh}}$ is the infimum of \eqref{eq:mcchh}
and $\bar{w}_h$ is a minimizer of \eqref{eq:mcch},
$\bar{u}_h$ the corresponding solution to \eqref{eq:pdeh},
and $L_w$ is the Lipschitz
constant of $j$ with respect to its second argument on $C$.

Moreover, 
\begin{gather}\label{eq:mcchh_approx_lb_ocp}
m_{\eqref{eq:ocp}} \ge m_{\eqref{eq:mcchh}} - L_u \|\bar{u} - \bar{u}_h\|_U
- L_w \|w_u - w_\ell\|_{L^\infty}^{\frac{1}{2}} \sqrt{d}^\frac{1}{2}
  \TV(\bar{w})^\frac{1}{2} h^\frac{1}{2},
\end{gather}
where $m_{\eqref{eq:ocp}}$ is the infimum of \eqref{eq:ocp},
$(\bar{u},\bar{w})$ is a minimizer of \eqref{eq:ocp},
and $L_u$ is the Lipschitz constant of $j$ with respect to its first argument
on the feasible set of \eqref{eq:mcchh}.
\end{theorem}
\begin{proof}
The convexity and boundedness of the feasible set follow as in \cref{prp:ocp_mcc_basics}.
The respective feasibility relations follow by construction, the definitions of \eqref{eq:mcch}
and \eqref{eq:mcchh}, and the assumed properties of $C$. The existence of minimizers
for \eqref{eq:mcchh} follows with the same arguments as for \eqref{eq:mcch}
with the only difference being that the existence of a convergent subsequence in $H$
of the control inputs $w$ as part of a minimizing sequence can already be deduced from
the finite-dimensional ansatz.
Because the feasible set of \eqref{eq:mcch} is bounded, the desired constant $L_u$ exists.
To prove the approximate lower bound \eqref{eq:mcchh_approx_lb}, we deduce that, for all
$(u_{h}, w_h,z_{1}^h,\ldots,z_{N_h}^h)$ that are feasible for \eqref{eq:mcch}, the estimates
\begin{align*}
m_{\eqref{eq:mcchh}} &\le 
 j(u_h, P_h w_h) + \alpha \TVh(P_h w_h) \\
 &\le  j(u_h, w_h) + \alpha \TVh(P_h w_h) + L_w \|w_h - P_h w_h\|_{H} \\
   &\le j(u_h, w_h) + \alpha \TV(w_h) + L_w \|w_u - w_\ell\|_{L^\infty}^{\frac{1}{2}} \|w_h - P_h w_h\|_{L^1(\Omega)}^\frac{1}{2} \\
   &\le j(u_h, w_h) + \alpha \TV(w_h) 
     + L_w \|w_u - w_\ell\|_{L^\infty}^{\frac{1}{2}} \sqrt{d}^\frac{1}{2} \TV(w_h)^\frac{1}{2} h^\frac{1}{2},
\end{align*}
hold, where we have used \cref{ass:mcchh} and the Lipschitz continuity
of $j$ with respect to the second argument for the second inequality and (the proof of)
(12.24) in Theorem 12.26 in \cite{maggi2012sets} for the third inequality.

To prove the approximate lower bound \eqref{eq:mcchh_approx_lb_ocp}, we chain this
estimate with the arguments from \cref{thm:mcch_approx_lb_on_ocp}. Let $(\bar{u},\bar{w})$
be a minimizer to \eqref{eq:ocp}, and let $\bar{w}_{h} = \bar{w}$.
Then \cref{ass:mcch} \ref{itm:phwphu_pdesolexists} gives a unique solution
$\bar{u}_h$ to \eqref{eq:pdeh}, and we choose $\bar{z}_h = (P_h u_h)(P_h \bar{w})$.
\Cref{ass:mcch} \ref{itm:richness_of_feasible_set} gives the required feasibility,
and using $P_h\bar{w}$ as well as $(\bar{u}_h,P_h \bar{w},\bar{z}_h^1,\ldots,\bar{z}_h^{N_h})$
in the estimate above gives
\begin{align}
m_{\eqref{eq:mcchh}}
&\le j(\bar{u}_h, \bar{w}) + \alpha \TV(\bar{w}) 
+ L_w \|w_u - w_\ell\|_{L^\infty}^{\frac{1}{2}}\sqrt{d}^\frac{1}{2} \TV(\bar{w})^\frac{1}{2} h^\frac{1}{2}
\nonumber \\
&\le j(\bar{u}, \bar{w}) + \alpha \TV(\bar{w})
+ L_u \|\bar{u}-\bar{u}_h\|_U
+ L_w \|w_u - w_\ell\|_{L^\infty}^{\frac{1}{2}} \sqrt{d}^\frac{1}{2} \TV(\bar{w})^\frac{1}{2} h^\frac{1}{2} ,
\label{eq:mcchh_ocp_inequality}
\end{align}
where the second inequality follows from the Lipschitz continuity of $j$
in the first argument.
\end{proof}
We now employ a $\Gamma$-convergence \cite{dal2012introduction} argument to prove that
the locally averaged McCormick relaxations \eqref{eq:mcchh} approximate \eqref{eq:mcc} for $h \searrow 0$.
This observation then implies that cluster points of sequences of minimizers to \eqref{eq:mcchh} for $h\searrow 0$
minimize \eqref{eq:mcc}. To this end, we need further compatibility conditions for the bounds on the control
and state variables.
\begin{assumption}\label{ass:mcchh_gamma_to_mcc}
In addition to \cref{ass:mcchh}, we assume the following.
\begin{enumerate}
\item\label{itm:bounded_eccentricity} Let $\{\calQ_h\}_h$ satisfy a uniform bounded eccentricity condition;
that is, there exists $C > 0$ such that for each $Q_i^h$ there is a ball $B_i^h$ such that
$Q_i^h \subset B_i^h$ and $|B_i^h| \le C |Q_i^h|$. See, for example, Definition 4.3 4.\ in \cite{manns2020multidimensional}.
\item\label{itm:TVh_compactness} Let $\sup_{h\searrow 0} \TVh(w_h) \le C$ 
and $\sup_{h\searrow 0} \|w_h\|_{L^1} \le C$ for some $C > 0$ imply
that there exists a subsequence $\{w_h\}_h$ and $w \in W$ such that
$w_h \to w$ in $L^1(\Omega)$.
\item\label{itm:TVh_gamma} Let $\TVh \circ P_h$ $\Gamma$-converge to $\TV$ for $h \searrow 0$.
\item\label{itm:ubnd_limits} Let $u_{\ell,h} \coloneqq \sum_{i=1}^{N_h} u_\ell^i \chi_{Q_h^i}$,
      $u_{u,h} \coloneqq \sum_{i=1}^{N_h} u_u^i \chi_{Q_h^i}$ satisfy
      $u_{\ell,h} \to u_\ell$ in $H$ and $u_{u,h} \to u_u$ in $H$.
\item\label{itm:wbnd_limits} Let $w_{\ell,h} \coloneqq \sum_{i=1}^{N_h} w_\ell^i \chi_{Q_h^i}$,
      $w_{u,h} \coloneqq \sum_{i=1}^{N_h} w_u^i \chi_{Q_h^i}$ satisfy
      $w_{\ell,h} \to w_\ell$ in $H$ and $w_{u,h} \to w_u$ in $H$.
\item\label{itm:uuh_approx} Let $u$ solve $Au + uw = f$ for some $w \in H$.
Then $u_h \to u$ in $U$ holds for the solutions $u_h$ to 
\eqref{eq:pdeh} for $h\searrow 0$.
\end{enumerate}
\end{assumption}
\begin{remark}
We note that Assumptions \ref{ass:mcchh_gamma_to_mcc}, \ref{itm:ubnd_limits}, 
and \ref{itm:wbnd_limits} can generally be ensured by inferring suitable 
bounds  $w_{\ell,h}$, $w_{u,h}$, $u_{\ell,h}$, $u_{u,h}$ from valid
bounds $w_\ell$, $w_u$, $u_\ell$, $u_u$. The other assumptions are quite typical
when approximating such problems or their solutions.
\end{remark}

\begin{proposition}\label{prp:gamma_mcchh_to_mcc}
Let \cref{ass:mcchh_gamma_to_mcc} hold. 
 Let the sequence
$\{(\bar{u}_{hh},\bar{w}_h^1,\ldots,\bar{w}_h^{N_h},\bar{z}_h^1,\ldots,\bar{z}_h^{N_h})\}_h$
solve the problems \eqref{eq:mcchh} for $h \searrow 0$.
Then the sequence $\{(\bar{u}_{hh},\bar{w}_{h},\bar{z}_{h})\}_h$ with
$\bar{w}_h \coloneqq \sum_{i=1}^{N_h} \bar{w}_h^i \chi_{Q^i_h}$ and
$\bar{z}_h \coloneqq \sum_{i=1}^{N_h} \bar{z}_h^i \chi_{Q^i_h}$ admits an
accumulation point $(\bar{u},\bar{w},\bar{z}) \in U \times W \times H$
such that for a subsequence (for ease of notation denoted by the same symbol), we obtain
\[ \bar{u}_{hh} \to \bar{u} \text{ in } U
\enskip\text{and}\enskip
   \bar{w}_h \weakstarto \bar{w} \text{ in } W
\enskip\text{and}\enskip
   \bar{z}_h \weakto \bar{z} \text{ in } H.
\]
If, in addition, for some $(u,w,z)$ minimizing \eqref{eq:mcc}, $(u_h,P_h w,P_h z)$ is feasible
for \eqref{eq:mcchh} as $h \searrow 0$, that is, $u_h = A^{-1}(P_h z - f)$, then the limit point $(\bar{u},\bar{w},\bar{z})$ minimizes \eqref{eq:mcc}.
\end{proposition}
\begin{proof}
The sequence $\{\bar{z}_h\}_h$ is bounded in $H$
by \cref{ass:mcchh_gamma_to_mcc} \ref{itm:ubnd_limits} and \ref{itm:wbnd_limits}
and the McCormick inequalities. Consequently, $\{\bar{u}_{hh}\}_h$ is also bounded in
$U$ because $A$ is continuously invertible. In turn, after possibly passing to a subsequence,
$\bar{z}_h \weakto \bar{z}$ holds for some $\bar{z}\in H$ and $\bar{u}_{hh} \to \bar{u}$
in $U$ for some $\bar{u} \in U$ because of the compact embedding $H \hookrightarrow^c U^*$
and the continuous invertibility of $A$. This argument also implies that the original
state equation $A\bar{u} + \bar{z} = f$ holds for the limit.

Let $u_{hh}$ solve \eqref{eq:pdeh} for $\bar{w}$. Then
$(u_{hh},P_h\bar{w},(P_h\bar{w})(P_hu_{hh}))$
is feasible for \eqref{eq:mcchh} by \cref{ass:mcch} \ref{itm:richness_of_feasible_set}.
The optimality of the tuples
$(\bar{u}_{hh},\bar{w}_h^1,\ldots,\bar{w}_h^{N_h},\bar{z}_h^1,\ldots,\bar{z}_h^{N_h})$
for the problems \eqref{eq:mcchh}
and \cref{ass:mcchh} give
\[ j(\bar{u}_{hh}, \bar{w}_h) + \alpha \TVh(\bar{w}_h)
\le j(u_{hh}, P_h \bar{w}) + \alpha \TVh(P_h \bar{w})
\le j(u_{hh}, P_h \bar{w}) + \alpha \TV(\bar{w}).
\]
Because $\{\bar{w}_h\}_h$, $\{P_h \bar{w}\}_h$, and in turn $\{u_{hh}\}_h$ are bounded
and $j$ is Lipschitz and thus bounded on bounded sets, this implies
that $\{\TVh(\bar{w}_h)\}_h$ is bounded in $[0,\infty)$.

The properties of the set $C$ imply that $\{\|\bar{w}_h\|_{H}\}_h$ is bounded so that
$\{\|\bar{w}_h\|_{L^1}\}_h$ is bounded in $\R$, too.
Thus, after passing to a subsequence (for ease of notation denoted by the same 
symbol), \cref{ass:mcchh_gamma_to_mcc} \ref{itm:TVh_compactness}
implies $\bar{w}_h \to \bar{w}$ in $L^1(\Omega)$ for some $\bar{w} \in W$.

The pointwise a.e.\ bounds on $P_h \bar{w}_h$ and $P_h \bar{u}_h$ in the 
McCormick inequalities together with \cref{ass:mcchh_gamma_to_mcc}
\ref{itm:ubnd_limits} and \ref{itm:wbnd_limits} imply boundedness of the
$\bar{z}_h$ in $H$. 

We are now concerned with feasibility of $(\bar{u}, \bar{w}, \bar{z})$ for \eqref{eq:mcc}.
The inclusion $\bar{w} \in C$ follows from the assumed properties of $C$.
For the bounds on the state variable, we first observe
\[ \|\bar{u} - P_h \bar{u}_{hh}\|_H 
\le \|\bar{u} - P_h\bar{u}\|_H 
  + \|P_h\|_{H,H}\|\bar{u} - \bar{u}_{hh}\|_H
\to 0
\]
from the triangle inequality, the nonexpansiveness of $P_h$, and Lebesgue's differentiation
theorem, which requires \cref{ass:mcchh_gamma_to_mcc} \ref{itm:bounded_eccentricity}.
Because, in addition, $u_{\ell,h} \to u_\ell$ and $u_{u,h} \to u_u$ in $H$ hold,
we can choose a further (subsubsub)sequence such that all three parts of the inequalities
$u_{\ell,h} \le P_h \bar{u}_{hh} \le u_{u,h}$ converge pointwise a.e. 
In turn, we obtain the pointwise a.e.\ inequalities
$u_\ell \le \bar{u} \le u_u$ for the limit.

It remains to show the feasibility of $\bar{z}$ for the four pointwise McCormick inequalities
in \eqref{eq:mcc} and the state equation. We  prove the claim only for the first one since the
others follow with an analogous argument. The feasibility of the solutions
for \eqref{eq:mcchh} gives
\[ \bar{z}_h \ge u_{\ell,h}\bar{w}_h + \bar{u}_{hh} w_{\ell,h} - u_{\ell,h}w_{\ell,h} 
\enskip\text{ a.e.\ in }  \Omega.
\]
From what has been shown to this point, after passing to an appropriate 
subsequence, the left-hand side converges weakly to $\bar{z}$, and the
right-hand side converges in $H$ and thus also in $L^1(\Omega)$. We apply 
\cref{lem:lower_bound_for_weak_limit} in order to obtain the desired
inequality
\[ \bar{z} \ge u_{\ell}\bar{w} + \bar{u} w_\ell - u_{\ell}w_{\ell}\enskip\text{ a.e.\ in }
\Omega. \]

It remains to show that $(\bar{u},\bar{w},\bar{z})$ minimizes \eqref{eq:mcc}.
We prove this claim by using a $\Gamma$-convergence-type argument. We need
to show $\liminf$- and $\limsup$-inequalities for the objectives when the 
iterates are restricted to the feasible sets.

For the $\liminf$-inequality, consider the
assumed optimal $(u,w,z)$ for \eqref{eq:mcc}
with $(u_h,P_h w,P_h z)$ being feasible for \eqref{eq:mcch}.
\cref{ass:mcchh_gamma_to_mcc} \ref{itm:bounded_eccentricity}
and Lebesgue's differentiation theorem give
$P_h z \to z$ in $H$. Consequently, the continuity of $A^{-1}$, $j$
and \cref{ass:mcchh_gamma_to_mcc} \ref{itm:TVh_gamma} give
\[ j(u,w) + \alpha \TV(w) = \lim_{h\searrow 0} j(u_{h}, P_h w) + \alpha\TVh(P_h w). \]
Combining the $\liminf$- and the $\limsup$-inequality yields that $(\bar{u}, \bar{w}, \bar{z})$
as a limit of minimizers for \eqref{eq:mcchh} minimizes 
\eqref{eq:mcc}; see, e.g., the proof of Proposition 1.18 in \cite{braides2002gamma}.
\end{proof}

As noted above, such a result requires a modification of the total variation functional dependent on $h$
for dimensions larger than one. In the one-dimensional case it is possible prove that $P_h$ is
nonexpansive with respect to the $\TV$-seminorm, which in turn implies the desired $\Gamma$-convergence
in this case without any modification of $\TV$ beforehand. This is shown below.
\begin{lemma}\label{lem:1DTVPh_gamma}
Let \cref{ass:mcch} hold.
Let $\Omega = (a,b)$ for some $a < b$. Let $f \in \BV(\Omega)$. Then
\[ \TV(P_h f) \le \TV(f). \]
In particular, 
\[ \TV \circ P_h\enskip\text{$\Gamma$-converges to}\enskip \TV. \]
\end{lemma}
\begin{proof}
The first claim can be shown by considering the intervals $Q_i \in \calQ_h$ one by one and using the
equivalence of $\TV$ to the pointwise variation when a \emph{good representative} is chosen;
see Definition 3.26, (3.24), and Theorem 3.28 in \cite{ambrosio2000functions}.

Let $f^n \to f$ in $L^1$. Because $P_h$ is nonexpansive with respect to the $L^1$-norm,
we have $\|P_h f - P_h f^n\|_{L^1} \le \|f - f^n\|_{L^1}$. Moreover, $P_h f \to f$ in $L^1(\Omega)$
holds by means of the Lebesgue differentiation theorem. In combination, we 
obtain $P_h f^n \to f$ in $L^1(\Omega)$. Consequently, the $\liminf$-inequality follows from the lower semi-continuity of $\TV$.
The $\limsup$-inequality follows from the first claim by choosing the constant sequence that
has $f$ in every element.
\end{proof}

\subsection{Optimization-based bound tightening}\label{sec:obbt}
While \cref{prp:gamma_mcchh_to_mcc} shows that solutions to the approximate relaxations
\eqref{eq:mcchh} of \eqref{eq:ocp} approximate solutions to the true relaxation \eqref{eq:mcc},
it is of higher practical importance that the optimal objective values of \eqref{eq:mcch} 
and \eqref{eq:mcchh} are as large as possible in practice. The only way to influence this are
the choices of the bounds $u_{\ell}^i$ and $u_{u}^i$,
which should be as tight as possible in order
to have the smallest possible feasible set and in turn the largest possible values for the
optimal objectives of \eqref{eq:mcch} and \eqref{eq:mcchh}
such that \eqref{eq:mcch_approx_lb_ocp} and \eqref{eq:mcchh_approx_lb_ocp} still hold.

A powerful albeit computationally expensive technique in MI(N)LP is optimization-based
bound tightening (OBBT), which is based on the observation that a convex set of a relaxation 
can be reduced (tightened) when minimizing and maximizing each variable in it over the
set and then intersecting the set with this optimized bound on said variable.
This is often used for variables arising in McCormick relaxations, see, for example, 
\cite{quesada1993global,quesada1995global,coffrin2015strengthening,bynum2018tightening,sundar2023optimization}.
Further references from more general vantage points on OBBT
are \cite{maranas1997global,puranik2017domain} and especially \cite{gleixner2017three} regarding
efficient techniques.

Inspecting the arguments in the proof of \cref{thm:mcchh_approx_lb_on_ocp} that lead to \eqref{eq:mcchh_approx_lb_ocp} (\eqref{eq:mcchh_ocp_inequality} in the proof), and similarly in \cref{thm:mcch_approx_lb_on_ocp}, we observe that these arguments remain valid when the feasible set of \eqref{eq:mcchh} is shrunk as long as \cref{ass:mcch} \ref{itm:richness_of_feasible_set} is preserved. Then,
the \emph{typical} argument for OBBT can be applied and the infimum $m_{\eqref{eq:mcchh}}$ increases when the feasible set is shrunk, while the preservation of \cref{ass:mcch} \ref{itm:richness_of_feasible_set} yields that the approximation error bound
that needs to be  subtracted in order to obtain a valid bound on $m_{\eqref{eq:ocp}}$ remains unaffected from this change. Consequently, the lower bound on
$m_{\eqref{eq:ocp}}$ is improved by this procedure. 

Before proving this, we show the desired property that
\cref{ass:mcch} \ref{itm:richness_of_feasible_set} is conserved
when optimizing the bounds $u_\ell^i$, $u_u^i$.

\begin{lemma}\label{lem:mcch_obbt_preserves_important_feasible_points}
	Let \cref{ass:mcch} hold.
	Let $h$ be fixed. Let $i \in \{1,\ldots,N_h\}$ be fixed. Then
	\begin{gather}\label{eq:obbt_problem_mcch}
	\tilde{u}^i_\ell\ /\ \tilde{u}^i_u 
	\coloneqq \min / \max\ \{ (P_h u)|_{Q_i} \,|\, (u,w,z_1,\ldots,z_{N_h})
	\text{ is feasible for } \eqref{eq:mcch} \}.
	\end{gather}
	are well defined.
	
	Moreover, $\tilde{u}_\ell^i \le (P_hu)|_{Q_i} \le \tilde{u}_u^i$
	holds for all $w \in C \cap [w_\ell,w_u]$ and $u \in U$ that solve
	\eqref{eq:pdeh}. In particular, \cref{ass:mcch} holds if $u^i_\ell$ is replaced by $\tilde{u}^i_\ell$ and $u^i_u$ is replaced by $\tilde{u}^i_u$.
\end{lemma}
\begin{proof}
We note that $u \mapsto (P_h u)|_{Q_i}$ is a weakly continuous operation from $H$ to $\R$.
Moreover, using the arguments for the existence of solutions to \eqref{eq:mcch} in the proof of
\cref{thm:mcch_approx_lb_on_ocp}, we obtain that the feasible set of \eqref{eq:mcch} is nonempty
and sequentially compact with respect to weak convergence of $u$ in $H$, weak convergence 
$w$ in $H$, and convergence of $z$ in $H$ (the last one has a finite-dimensional ansatz).
Consequently, \eqref{eq:obbt_problem_mcch} has a solution, and $\tilde{u}^i_\ell$ is well defined.

Inspecting \cref{ass:mcch} shows that it remains to prove that \cref{ass:mcch}
\ref{itm:richness_of_feasible_set} stays valid when replacing $u_\ell^i$ by $\tilde{u}^i_\ell$.
For all $w \in C$ and $u \in U$ that solve $Au + (P_h u)(P_h w) = f$, we define
$z_i \coloneqq (P_h u)|_{Q^i_h}(P_h w)|_{Q^i_h}$ for $i \in \{1,\ldots,N_h\}$.
Inspecting \eqref{eq:mcch}, we obtain that $(u, w, z_1,\ldots,z_{N_h})$ is feasible
for \eqref{eq:mcch} by means of \cref{ass:mcch} \ref{itm:richness_of_feasible_set}.
Consequently, all $w \in C$ and $u \in U$ that solve $Au + (P_h u)(P_h w) = f$ satisfy
\begin{align*}
\tilde{u}_\ell^i &=
\min \{ (P_h u)|_{Q_i} \,|\, (u,w,z_1,\ldots,z_{N_h})
\text{ is feasible for } \eqref{eq:mcch} \}\\
&\le \inf \{ (P_h u)|_{Q_i} \,|\, (u,w) \in U \times C, Au + (P_h u)(P_h w) = f \},
\end{align*}
and thus also $\tilde{u}_\ell^i \le (P_h u)|_{Q_i}$, which proves the first claim.

The second well-definedness and claim follow analogously.
\end{proof}
A similar argument can be made for \eqref{eq:mcchh}.
\begin{lemma}\label{lem:mcchh_obbt_preserves_important_feasible_points}
	Let \cref{ass:mcchh} hold.
	Let $h$ be fixed. Let $i \in \{1,\ldots,N_h\}$ be fixed. Then
	\begin{gather*}\label{eq:obbt_problem}
	\tilde{u}^i_\ell\ /\ \tilde{u}^i_u \coloneqq
	\min / \max 
	\{ (P_h u)|_{Q_i} \,|\, (u,w_1,\ldots,w_{N_h},z_1,\ldots,z_{N_h})
	\text{ is feasible for } \eqref{eq:mcchh} \}.
	\end{gather*}
	are well defined.
	
	Moreover, $\tilde{u}_\ell^i \le (P_hu)|_{Q_i} \le \tilde{u}_u^i$
	holds for all $w \in C \cap [w_\ell,w_u]$ and $u \in U$.
	In particular, \cref{ass:mcchh} holds if $u^i_\ell$ is replaced by $\tilde{u}^i_\ell$ and $u^i_u$ is replaced by $\tilde{u}^i_u$.	
\end{lemma}
\begin{proof}
The proof parallels the one of \cref{lem:mcch_obbt_preserves_important_feasible_points}.
\end{proof}

We can now state a bound-tightening procedure that successively computes new bounds
in order to improve the approximate lower bounds \eqref{eq:mcch_approx_lb_ocp}
or \eqref{eq:mcchh_approx_lb_ocp} respectively in \cref{alg:abstract_obbt}.
\begin{algorithm}[t]
	\caption{Optimization-based bound tightening (OBBT) for \eqref{eq:mcch} (or \eqref{eq:mcchh})}\label{alg:abstract_obbt}
	\textbf{Input:} Feasible set $\calF^0$ of \eqref{eq:mcch} (or \eqref{eq:mcchh}).
	\begin{algorithmic}[1]
		\For{$n = 1,\ldots$}
		\State Choose $i \in \{1,\ldots,N_h\}$.
		\State Choose $s \in \{\ell,u\}$.
		\If{$s = \ell$}
			\State $\tilde{u}^i_{\ell} \gets \min\{ (P_h u_h)|_{Q_i} \,|\, (u_h,w_h,z_h) \in \calF^{n-1} \}.$\label{ln:set_new_lb}
		\Else
			\State $\tilde{u}^i_{u} \gets \max\{ (P_h u_h)|_{Q_i} \,|\, (u_h,w_h,z_h) \in \calF^{n-1} \}.$\label{ln:set_new_ub}
		\EndIf
		\State $\calF^n \gets \calF^{n - 1}$ with $u^i_s$ being replaced by $\tilde{u}^i_s$.
		\EndFor
	\end{algorithmic}
\end{algorithm}
Using
\cref{lem:mcch_obbt_preserves_important_feasible_points,lem:mcchh_obbt_preserves_important_feasible_points},
we can prove our main bound-tightening results for \eqref{eq:mcch} and \eqref{eq:mcchh}.
\begin{theorem}\label{thm:obbt_mcch}
Let \cref{ass:mcch} be satisfied. Let $C$ be a weakly sequentially compact subset of $H$.
Then \cref{alg:abstract_obbt} executed with the feasible set of \eqref{eq:mcch} as $\calF^0$
induces a sequence of optimization problems
\begin{gather}\label{eq:mcchn}
\min_{u, w,z_1,\ldots,z_{N_h}}\ j(u, w) + \alpha\TV(w)
\quad\text{\emph{s.t.}}\quad
(u, w, z_1,\ldots,z_{N_h}) \in \calF^n
\tag{McC$_h^n$}
\end{gather}
that satisfy \eqref{eq:mcch_approx_lb_ocp} for all $n \in \N$.
Moreover, the sequence of optimal objective values $(m_{\eqref{eq:mcchn}})_n$ 
is monotonically nondecreasing.
\end{theorem}
\begin{proof}
By assumption, the prerequisites of \cref{lem:mcch_obbt_preserves_important_feasible_points}
are satisfied in the first iteration and thus hold inductively for all iterations $n \in \N$.
This implies that \eqref{eq:mcch_approx_lb_ocp} holds for all problems $\eqref{eq:mcchn}$
because \cref{thm:mcch_approx_lb_on_ocp} asserts \eqref{eq:mcch_approx_lb_ocp} under \cref{ass:mcch}.
Because the feasible set is always a subset of the previous one, the infima are monotonically nondecreasing.
\end{proof}

\begin{theorem}\label{thm:obbt_mcchh}
Let \cref{ass:mcchh} be satisfied. Then
\cref{alg:abstract_obbt} executed with the feasible set of \eqref{eq:mcchh} as $\calF^0$
induces a sequence of optimization problems
\begin{gather}\label{eq:mcchhn}
\begin{aligned}
\min_{\substack{u, w_1,\ldots,w_{N_h},\\ z_1,\ldots,z_{N_h}}}\
&j\left(u,\sum_{i=1}^{N_h}\chi_{Q_h^i}w_i\right) 
+ \alpha\TV\left(\sum_{i=1}^{N_h}\chi_{Q_h^i}w_i\right) \\
\text{s.t.}\ \ \ \
&(u, w_1,\ldots,w_{N_h} , z_1,\ldots,z_{N_h}) \in \calF^n
\end{aligned}\tag{McC$_{hh}^n$}
\end{gather}
that satisfy \eqref{eq:mcchh_approx_lb} for all $n \in \N$.
Moreover, the corresponding sequence of optimal
objective values $\{m_{\eqref{eq:mcchhn}}\}_n$ is monotonically nondecreasing.
\end{theorem}
\begin{proof}
The proof parallels the one of \cref{thm:obbt_mcchh} with the applications of
\cref{lem:mcch_obbt_preserves_important_feasible_points} and
\cref{thm:mcch_approx_lb_on_ocp} being replaced by
applications of \cref{lem:mcchh_obbt_preserves_important_feasible_points}
and \cref{thm:mcchh_approx_lb_on_ocp}.
\end{proof}

\section{Application to an elliptic optimal control problem}\label{sec:elliptic_example}
We now apply our theoretical considerations from \cref{sec:mccormick} to an instance
of \eqref{eq:ocp} with an elliptic PDE that is defined on the one-dimensional domain
$\Omega = (0,1)$. To this end, let
$C \coloneqq \{ w \in L^\infty(0,1) \,|\, w_\ell \le w(x) \le w_u \text{ a.e.}\}$
and $f \in H$ for real constants $- \pi^2 < w_\ell < w_u < \pi^2$ be fixed.

For these assumptions, we outline the PDE setting in \cref{sec:pdesetting} and give ellipticity
estimates as well as bounds from above on the norm with constants that are as sharp as we were
able to obtain them from the literature. We provide the objective in \cref{sec:objective}.
We then verify \cref{ass:mcc,ass:mcch,ass:mcchh,ass:mcchh_gamma_to_mcc}
in \cref{sec:verification_of_assumptions}.
We prove an a priori estimate for $\|u - u_h\|_U$ in the approximate lower bounds \eqref{eq:mcch_approx_lb_ocp}, \eqref{eq:mcchh_approx_lb_ocp}
in \cref{sec:apriori}. We prove differentiability properties for the control-to-state
operator of the PDE in \cref{sec:control_to_state_operator}. In particular,
we verify the assumptions imposed in \cite{leyffer2022sequential}.

\subsection{PDE setting}\label{sec:pdesetting}
For $w \in C$, we are interested in solutions $u \in U \coloneqq H^1_0(0,1)$ to the PDE in weak form
\begin{gather}\label{eq:pde}
\begin{aligned}
\int_0^1 \nabla u \nabla v + \int_0^1 w u v &= \int_0^1 f v  \quad \text{ for all } v \in U.
\end{aligned}
\end{gather}
The operator form is
\[ -\Delta u + uw = f \text{ on } (0,1)
\text{ with } u(0) = u(1) = 0,
\]
that is, $A = -\Delta$ with homogeneous Dirichlet boundaries and $N(u,w) = uw$.
We define the bilinear form $B : U \times U \to \R$ by
\[ B(u,v) \coloneqq \int_0^1 \nabla u \nabla v + \int_0^1 u w v \]
and obtain
\begin{gather}\label{eq:coercivity}
B(u,u) \ge \|\nabla u\|_{H}^2 + w_\ell \|u\|_{H}^2
\quad\text{ for all }u \in U.
\end{gather}
The existence of (unique) solutions to \eqref{eq:pde} follows from the Lax--Milgram lemma if 
\[ \|\nabla u\|_{H}^2 + w_\ell \|u\|_{H}^2 \ge c_1(w_\ell) \|\nabla u\|_{H}^2 \]
holds for some $c_1(w_\ell) > 0$, which is true if $w_\ell \ge 0$. For the case
$w_\ell < 0$, we recall embedding constants for the Sobolev inequalities
for the one-dimensional domain $(0,1)$:
\begin{align}
\pi^2 \|u\|_{H}^2 &\le \|\nabla u\|_{H}^2\text{ and}\label{eq:H10embedL2} \\
4 \|u\|_{L^\infty}^2 &\le \|\nabla u\|_{H}^2\label{eq:H10embedLinf}.
\end{align}
The constant in \eqref{eq:H10embedL2} is optimal and, for example, given in (1.2)
in \cite{veeser2012poincare} (see \cite{payne1960optimal} for a proof). The 
constant in \eqref{eq:H10embedLinf} can be found in (13) in \cite{schmidt1940ungleichung}
(choose $a = \infty$, $b = 2$, $s = 1$ therein). By \eqref{eq:H10embedL2}, we obtain that $0 < c_1(w_\ell) \coloneqq 1 + w_\ell \pi^{-2}$ because (by our choice above)
\begin{gather}\label{eq:lb_well}
w_\ell > -\pi^2.
\end{gather}
The Lax--Milgram lemma also yields that the solution to \eqref{eq:pde} satisfies
\begin{align}
\|\nabla u\|_{H} &\le \frac{1}{c_1(w_\ell)}\|f\|_{H}\text{ and}\label{eq:H10bound} \\
\|u\|_{L^\infty} &\le \frac{1}{2 c_1(w_\ell)}\|f\|_{H},\label{eq:Linfbound}
\end{align}
where the second inequality follows from \eqref{eq:H10embedLinf}.
We observe that $u$ is also the weak solution of the PDE
$\int_0^1 \nabla u \nabla v = \int_0^1 g v$ for  all $v \in U$
with the choice $g = f - w u $. Because $g \in H$, this gives the improved
regularity $u \in H^2(0,1)$; see, for example, Theorem 9.53 in \cite{renardy2006introduction}.

\subsection{Objective}\label{sec:objective}
We consider a tracking-type objective functional $j : U \times W \to \R$ that is defined
for $(u,w) \in U \times W$ as
\begin{gather}\label{eq:juw}
j(u,w) \coloneqq \frac{1}{2}\|u - u_d\|_{H}^2.
\end{gather}
Since our specific choice for $j$ does not depend on $w$, we will abbreviate it as $j(u)$
in the remainder. If the norm of the input is bounded by a constant $r_u$, we
obtain that a feasible Lipschitz constant of $j$ on the ball $\overline{B_{r_u}(0)}$ is given
by $L_u \coloneqq r_u + \|u_d\|_H$. Clearly, $L_w = 0$.

\subsection{Verification of assumptions}\label{sec:verification_of_assumptions}
We now verify \cref{ass:mcc,ass:mcch,ass:mcchh,ass:mcchh_gamma_to_mcc} one by one.

\paragraph{\Cref{ass:mcc}}
Let $c_2(w_\ell,w_u) \coloneqq 1 - \max\{|w_\ell|,|w_u|\}\pi^{-2}$.
We choose $u_u = \frac{1}{2 c_2(w_\ell,w_u)}\|f\|_H$
and $u_\ell = -\frac{1}{2 c_2(w_\ell,w_u)}\|f\|_H$.
Then $c_2(w_\ell,w_u) \le c_1(w_\ell)$ 
gives $|u_u| = |u_\ell| \ge \frac{1}{2 c_1(w_\ell)}\|f\|_H$
so that together with the bounds $w_\ell$, $w_u$, we obtain
that \cref{ass:mcc} is satisfied because of \eqref{eq:Linfbound}.
(We could of course choose the tighter bounds
$u_u = \frac{1}{2c_1(w_\ell)}\|f\|_H$ and $u_\ell = -\frac{1}{2c_1(w_\ell)}\|f\|_H$ here,
but the relaxed bounds using $c_2(w_\ell,w_u)$ instead of $c_1(w_\ell)$
will be used to assert the other assumptions below.)

\paragraph{\Cref{ass:mcch}}
We consider a sequence of uniform partitions $\{\calQ_{h_n}\}_{n \in \N}$
of the domain $(0,1)$ into $2^{n + 1}$ intervals with $h_n = 2^{-1 - n}$ for $n \in \N$
and $N_{h_n} = 2^{n + 1}$, which gives \cref{ass:mcch} \ref{itm:partitions}.

Let $n \in \N$, and abbreviate $h \coloneqq h_n$. For each interval
$I \in \calQ_{h}$, the projection $P_I$ is defined as $P_I g \coloneqq \chi_I \frac{1}{h}\int_I g$
for $g \in L^1(0,1)$. Clearly, $P_I$ is nonexpansive in $L^1(0,1)$, $H$, and $L^\infty(0,1)$; that is,
\[ \|P_I g\|_{L^\infty} \le \|g\|_{L^\infty}\text{ and }\|P_I g\|_{H} \le \|g\|_{H}. \]
The linear projection operator $P_h$ from \eqref{eq:proj_dg0} satisfies
$P_h(g) = \sum_{I \in \calQ_h} P_I g$ and
\begin{gather}\label{eq:proj_H_estimate}
\|g - P_h g\|_{H} \le \pi h \|\nabla g\|_{H},
\end{gather}
where the constant $\pi$ is due to the application of the one-dimensional 
Poincar\'{e} inequality into the interpolation estimate for piecewise
constant functions; see, for example, \S 3.1 in \cite{ciarlet2002finite}.
From (12.24) in Theorem 12.26 in \cite{maggi2012sets}, we also obtain
\begin{gather}\label{eq:proj_L1_estimate}
\|w - P_h w\|_{L^1} \le h \TV(w),
\end{gather}
which we have already used in the proof of \cref{thm:mcchh_approx_lb_on_ocp}.

For arbitrary, fixed $w$ with $w(x) \in [w_\ell,w_u]$ a.e, we define the bilinear form $B_h : U \times U \to \R$ as
\[ B_h(u,v) \coloneqq \int_0^1 \nabla u \nabla v + \int_0^1 (P_h w) (P_h u) v \]
for $u$, $v \in U$ and obtain the coercivity
\[ B_h(u,u) \ge \|\nabla u\|_{H}^2 - \max\{|w_\ell|,|w_u|\}\|u\|_{H} \|P_h(u)\|_{H}
\ge \|\nabla u\|_{H}^2 - \max\{|w_\ell|,|w_u|\}\|u\|_{H}^2.
\]
Consequently, the existence of unique solutions to the PDE
\begin{gather}\label{eq:pde_h}
\begin{aligned}
\int_0^1 \nabla u  \nabla v
+ \int_0^1 (P_h w)(P_h u) v &= \int_0^1 f v  \quad \text{ for all } v \in U.
\end{aligned}
\end{gather}
follows from the Lax--Milgram lemma with analogous estimates to \eqref{eq:H10bound} and \eqref{eq:Linfbound}, where $c_1(w_\ell)$ is replaced by $c_2(w_\ell,w_u) = 1 - \max\{|w_\ell|,|w_u|\}\pi^{-2}$
because $0 < c_2(w_\ell,w_u)$. Specifically, we have
\begin{align}
\|\nabla u\|_{H} &\le \frac{1}{c_2(w_\ell,w_u)}\|f\|_{H}\text{ and}\label{eq:H10bound_locavg} \\
\|u\|_{L^\infty} &\le \frac{1}{2 c_2(w_\ell,w_u)}\|f\|_{H},\label{eq:Linfbound_locavg}.
\end{align}
This proves \cref{ass:mcch} \ref{itm:phwphu_pdesolexists}. As for \eqref{eq:pde}, we observe that $u$ solves
$\int_0^1 \nabla u \nabla v = \int_0^1 g v$ for  all $v \in U$ with the choice $g = f - (P_h w)(P_h u)$.
Because $g \in H$, this gives the improved regularity $u \in H^2(0,1)$; see, for example,
Theorem 9.53 in \cite{renardy2006introduction}.

Because $P_h$ is nonexpansive with respect to $L^\infty(0,1)$, we can reuse the bounds $w_\ell$ and
$w_u$ for $w_\ell^i$ and $w_u^i$ for all $i \in \{1,\ldots,N_h\}$.
With the same argument as for \cref{ass:mcc}, we choose $u_u^i = \frac{1}{2 c_2(w_\ell,w_u)}\|f\|_H$
and $u_\ell^i = -\frac{1}{2 c_2(w_\ell,w_u)}\|f\|_H$ for all $i \in \{1,\ldots,N_h\}$. In combination, \cref{ass:mcch}
\ref{itm:richness_of_feasible_set} is satisfied.

\paragraph{\Cref{ass:mcchh}}
In our one-dimensional setting, \cref{lem:1DTVPh_gamma} implies that \cref{ass:mcchh} holds
with the choice $\TVh \coloneqq \TV$.

\paragraph{\Cref{ass:mcchh,ass:mcchh_gamma_to_mcc}}
\Cref{ass:mcchh_gamma_to_mcc} \ref{itm:bounded_eccentricity} is satisfied because of the
uniform discretization of the computational domain $\Omega = (0,1)$
into intervals. \Cref{ass:mcchh_gamma_to_mcc} \ref{itm:TVh_compactness} follows from the properties of
the total variation seminorm with the choice $\TVh \coloneqq \TV$; see Theorem 3.23 in \cite{ambrosio2000functions}.
\Cref{ass:mcchh_gamma_to_mcc} \ref{itm:TVh_gamma} follows from  \cref{lem:1DTVPh_gamma}.
\Cref{ass:mcchh_gamma_to_mcc} \ref{itm:ubnd_limits} follows because we have chosen
$u_\ell^i = u_\ell$ and $u_u^i = u_u$ for all $i$ and $h$.
\Cref{ass:mcchh_gamma_to_mcc} \ref{itm:wbnd_limits} follows because we have chosen
$w_\ell^i = w_\ell$ and $w_u^i = w_u$ for all $i$ and $h$.
\Cref{ass:mcchh_gamma_to_mcc} \ref{itm:uuh_approx} follows
from \cref{lem:u_uh_estimate_H} that is proven in \cref{sec:apriori}
and a bootstrapping argument /
the continuous invertibility of $A$.

\subsection{Estimate on $\|u - u_h\|_H$ and a priori estimates on $m_{\text{(\ref{eq:ocp})}}$}\label{sec:apriori}
For our example PDE setting and the tracking-type objective in $H$,
we are able to obtain lower bounds on \eqref{eq:mcch_approx_lb_ocp} and 
\eqref{eq:mcchh_approx_lb_ocp} that are quadratic in $h$
because we can improve over $\|u - u_h\|_U$ when considering $\|\cdot\|_H$
instead of $\|\cdot\|_U$. We first show the necessary estimates on $\|u - u_h\|_H$ and
subsequently prove the lower bounds on $m_{\eqref{eq:ocp}}$.
\begin{lemma}\label{lem:u_uh_estimate_H}
Let \cref{ass:mcch} hold. Let a mesh size $h$ be fixed.
Let $u$ solve \eqref{eq:pde} and $u_h$ solve \eqref{eq:pde_h} for the same fixed $w \in W$.
Then the estimates
\[ 
\|u - u_h\|_H
\le C_{3/2}^a(w_\ell,w_u,f,w) h^{\frac{3}{2}}
+ C_{3/2}^b(w_\ell,w_u,f) h^2 
\]
and
\begin{align}
\|u - u_h\|_H &\le C_2(w_\ell,w_u,f,w,u_h) h^2\label{eq:uuh_quadratic}
\end{align}
hold for constants $C_{3/2}^a(w_\ell,w_u,f,w)$, $C_{3/2}^b(w_\ell,w_u,f)$,
$C_2(w_\ell,w_u,f,w,u_h) > 0$.
\end{lemma}
\begin{proof}
We observe that the local averaging of $w$ and $u$ in the lower-order term in \eqref{eq:pde_h} satisfies
Galerkin-type orthogonality properties. Specifically, for given $\phi$, $\psi$, $\theta \in H$, we have
\begin{gather}\label{eq:orthogonality}
\int_0^1 (\phi - P_h \phi)(P_h \psi)(P_h \theta) 
= \sum_{i=1}^{N_h} 
\left(\frac{1}{|Q^i_h|} \int_{Q^i_h} \psi\right) 
\left( \frac{1}{|Q^i_h|} \int_{Q^i_h}\theta\right)
\underbrace{\int_{Q^i_h} \phi - P_h \phi}_{= 0}
= 0.
\end{gather}
Thus, we use an Aubin--Nitsche duality argument in order to obtain an estimate on the right-hand side of
\begin{gather}\label{eq:dual_norm}
\|u - u_h\|_H = \sup_{g \neq 0} \frac{(u - u_h, g)_H}{\|g\|_H}.
\end{gather}
To this end, we consider the solution $p \in U$ to the adjoint PDE
\begin{gather}\label{eq:orthogonality_adjoint}
 \int_0^1 \nabla p \nabla v + \int_0^1 p w v = \int_0^1 g v\quad\text{ for all } v \in U
\end{gather}
for arbitrary $g \in H$, which has the same properties as \eqref{eq:pde}
(the only difference is the source term $g$ instead of $f$). Moreover, we observe with the choice $v = p$
in \eqref{eq:pde} and \eqref{eq:pde_h} and the insertion of a suitable zero that
\begin{gather}\label{eq:uuh_diff_p_as_test}
\int_0^1 \nabla(u - u_h) \nabla p = -\int_0^1 (u - P_h u_h) w p - \int_0^1 (P_h u_h)(w - P_h w)p.
\end{gather}

We choose $v = u - u_h$ in \eqref{eq:orthogonality_adjoint} and apply \eqref{eq:uuh_diff_p_as_test} so that we obtain
\begin{align*}
(u - u_h, g)_H &= \int_0^1 \nabla p\nabla(u - u_h) + \int_0^1 p w (u - u_h)\\
&= \int_0^1 (u - u_h) w p - \int_0^1 (u - P_h u_h)  w p - \int_0^1 (P_h u_h)(w - P_h w)p.
\end{align*}
Combining the first two terms on the right-hand side and subtracting
$0 = \int_0^1 (P_h u_h)(w - P_h w)(P_h p)$,
which holds because of \eqref{eq:orthogonality}, we obtain
\begin{align*}
(u - u_h, g)_H
&= \int_0^1 (P_h u_h - u_h)  w p
 - \int_0^1 (P_h u_h) (w - P_h w)(p - P_h p).
\end{align*}
Inserting another zero and applying $0 = \int_0^1 (P_h u_h)(w - P_h w)(P_h p)$ again give
\begin{align*}
(u - u_h, g)_H 
&= \int_0^1 (P_h u_h - u_h) (w - P_h w) p
+ \int_0^1 (P_h u_h - u_h) (P_h w) (p - P_h p) \\
&\quad - \int_0^1 (P_h u_h) (w - P_h w)(p - P_h p).
\end{align*}
Now, we estimate the terms on the right-hand side and obtain with the same estimates
that have been used in the preceding subsections the following estimates:
\begin{align*}
\hspace{2em}&\hspace{-2em}\int_0^1 (P_h u_h - u_h) (w - P_h w) p\\
&\le h \pi \|\nabla u_h\|_H \|w - P_hw\|_H \|p\|_{L^\infty} && \text{\scriptsize H\"older, \eqref{eq:proj_H_estimate}}\\
&\le h\frac{\pi}{c_2(w_\ell,w_u)}\|f\|_H |w_u - w_\ell|^{\frac{1}{2}}\|w - P_hw\|_{L^1}^{\frac{1}{2}} \frac{1}{2 c_1(w_\ell)}\|g\|_H
 && \text{\scriptsize H\"older, \eqref{eq:Linfbound}, \eqref{eq:H10bound_locavg} }\\
&\le \frac{\pi|w_u - w_\ell|^{\frac{1}{2}}}{2 c_1(w_\ell)c_2(w_\ell,w_u)} \TV(w)^{\frac{1}{2}}\|f\|_H h^{\frac{3}{2}} \|g\|_H,
&&\text{\scriptsize \eqref{eq:proj_L1_estimate}}
\end{align*}
\begin{align*}
\int_0^1 (P_h u_h - u_h) (P_h w) (p - P_h p) 
&\le \frac{\pi^2 \max\{|w_u|,|w_\ell|\}}{c_1(w_\ell) c_2(w_\ell,w_u)} \|f\|_H h^2 \|g\|_H,
&&\text{\scriptsize  H\"older, \eqref{eq:Linfbound}, \eqref{eq:proj_H_estimate}, \eqref{eq:H10bound_locavg}}
\end{align*}
and
\begin{align*}
\hspace{2em}&\hspace{-2em}\int_0^1 (P_h u_h) (w - P_h w)(p - P_h p)\\
&\le \frac{\pi|w_u - w_\ell|^{\frac{1}{2}}}{2 c_1(w_\ell)c_2(w_\ell,w_u)} \TV(w)^{\frac{1}{2}}\|f\|_H h^{\frac{3}{2}} \|g\|_H.
&&\text{\scriptsize  H\"older, \eqref{eq:Linfbound}, \eqref{eq:proj_H_estimate}, \eqref{eq:proj_L1_estimate}, \eqref{eq:H10bound_locavg}}
\end{align*}
Combining these considerations, we obtain from \eqref{eq:dual_norm}
\[ 
\|u - u_h\|_H
\le \underbrace{\frac{\pi|w_u - w_\ell|^{\frac{1}{2}}}{c_1(w_\ell)c_2(w_\ell,w_u)} \TV(w)^{\frac{1}{2}} \|f\|_H}_{\eqqcolon\,C_{3/2}^a(w_\ell,w_u,f,w)} h^{\frac{3}{2}}
 + \underbrace{\frac{\pi^2 \max\{|w_u|,|w_\ell|\}}{c_1(w_\ell) c_2(w_\ell,w_u)} \|f\|_H}_{\eqqcolon\,C_{3/2}^b(w_\ell,w_u,f)} h^2,
\]
which is the first of the claimed estimates.

We next prove the second estimate. Because we know that $u$, $u_h$, $p$ are also $H^2(0,1)$-functions
(see again Theorem 9.53 in \cite{renardy2006introduction}), their derivatives are uniformly bounded, and we can estimate them
pointwise. Let $(\eta_\varepsilon)_\varepsilon$ be a family of standard mollifiers. For $y \in (0,1)$, we define the antiderivative 
$N_\varepsilon^y(x) \coloneqq \int_{-\infty}^x \eta_\varepsilon(z - y) \dd z$. We test \eqref{eq:pde_h} with $N_\varepsilon^y$
and deduce
\begin{align*}
\int_0^1 \nabla u_h (x) \eta_\varepsilon(x - y) \dd x
&= \int_0^1 f(x) N_\varepsilon^y(x)\dd x -\int_0^1 (P_hu_h)(x) (P_hw)(x) N_\varepsilon^y(x)\dd x \\
&\le \int_0^1   |f(x) - (P_hu_h)(x) (P_hw)(x)| \int_{-\infty}^\infty \eta_\varepsilon(x - y)  \dd y \dd x \\
&= \|f - (P_hu_h) (P_hw) \|_{L^1}.
\end{align*}
Driving $\varepsilon \searrow 0$ and supremizing over the left-hand side, we obtain
\[ \|\nabla u_h\|_{L^\infty} \le \|f - (P_hu_h) (P_hw) \|_{L^1}. \]
An analogous argument and applications of the triangle inequality,
H\"older's inequality, \eqref{eq:H10embedL2}, and \eqref{eq:H10bound} give
\[ \|\nabla p\|_{L^\infty} \le \|g - p w\|_{L^1} \le \left(1 + \frac{\max\{|w_\ell|,|w_u|\}}{\pi c_1(w_\ell)}\right)\|g\|_H .
\]
We also obtain for $\phi \in W^{1,\infty}(0,1)$ that
\[ \|\phi - P_h \phi\|_{L^\infty} \le \frac{1}{2} h \|\nabla \phi\|_{L^\infty}, \]
where $\frac{1}{2}$ is the Wirtinger--Sobolev constant that can be found in (13)
in \cite{schmidt1940ungleichung} (choose $a = \infty$, $b = \infty$, $s = 1$ therein).

Then, using H\"{o}lder's inequality with different exponents, we can estimate the two $h^{\frac{3}{2}}$-terms similarly as above but we replace the H\"older conjugates
$(2,2)$ by $(1,\infty)$ in order to obtain $h^2$-terms, specifically
\begin{align*}
\int_0^1 (P_h u_h - u_h) (w - P_h w) p
&\le \frac{1}{2} h \|\nabla u_h\|_{L^\infty} \|w - P_hw\|_{L^1} \|p\|_{L^\infty}\\
&\le \frac{1}{4 c_1(w_\ell) c_2(w_\ell, w_u)} \TV(w) \|f - (P_hu_h) (P_hw) \|_{L^1} h^2 \|g\|_H 
\end{align*}
and
\begin{align*}
\int_0^1 (P_h u_h) (w - P_h w)(p - P_h p)
&\le \frac{c_1(w_\ell) + \max\{|w_\ell|,|w_u|\}\pi^{-1}}{4c_1(w_\ell)c_2(w_\ell,w_u))}\TV(w)\|f\|_H h^2\|g\|_H .
\end{align*}
In summary, we obtain
\begin{align*}
\|u - u_h\|_H &\le C_2(w_\ell,w_u,f,w,u_h)h^2 
\end{align*}
with
\begin{multline}\label{eq:C2}
C_2(w_\ell,w_u,f,w,u_h) \coloneqq \frac{1}{4c_1(w_\ell) c_2(w_\ell,w_u)}
\Big( 4\pi^2 \max\{|w_u|,|w_\ell|\} \|f\|_H +\\
 \TV(w) \|f - (P_hu_h) (P_hw) \|_{L^1}
 + \big(c_1(w_\ell) + \max\{|w_\ell|,|w_u|\}\pi^{-1}\big)\TV(w)\|f\|_H \Big),
\end{multline}
which is the second estimate.
\end{proof}
We now provide a lower bound on $m_{\eqref{eq:ocp}}$ in terms of $h^2$
for our guiding example.
\begin{lemma}\label{lem:improved_approx_lb_on_ocp}
Let $j_0 \ge 0$ be a lower bound on $j(\bar{u})$ if $(\bar{u},\bar{w})$ minimizes \eqref{eq:ocp}.
Let $\hat{w}\in W$ be feasible for \eqref{eq:ocp}. Then
\begin{gather}
m_{\eqref{eq:ocp}} \ge m_{\eqref{eq:mcchh}} - c_{quad}(w_\ell, w_u, f, \hat{w}, j_0, u_{\ell,h},u_{u,h})h^2\label{eq:quad_bound}.
\end{gather}
\end{lemma}
\begin{proof}
For the objective $j(u) = \tfrac{1}{2}\|u - u_d\|_H^2$, we obtain $\nabla_u j(u) = u - u_d$.
Since we have pointwise lower and upper bounds $u_\ell$ and $u_u$ on $u$, for example,
from \eqref{eq:Linfbound_locavg}, we can overestimate the Lipschitz constant of $j$ on the
feasible set with respect to the $L^2$-norm by setting
$\tilde{d}_u(x) \coloneqq \max\{|u_u(x) - u_d(x)|,|u_d(x) - u_\ell(x)|\}$ for a.a.\ $x \in (0,1)$.
Then the Lipschitz constant can be estimated as $L_u \le \|\tilde{d}_u\|_{H}$.

Using \eqref{eq:uuh_quadratic} from \cref{lem:u_uh_estimate_H}, we obtain
\begin{align*}
m_{\eqref{eq:ocp}} &\ge m_{\eqref{eq:mcchh}} 
- \|\tilde{d}_u\|_H C_2(w_\ell,w_u,f,\bar{w},\bar{u}_h) h^2.
\end{align*}
Inspecting $C_2(w_\ell,w_u,f,\bar{w},\bar{u}_h)$ in  \cref{lem:u_uh_estimate_H},
we observe that we can overestimate $\TV(\bar{w}) \le
\frac{j(\hat{u}) + \alpha \TV(\hat{w}) - j_0}{\alpha}$,
where $\hat{u}$ denotes the unique solution
to \eqref{eq:pde} for $\hat{w}$.
Moreover, using available bounds on $P_h\bar{u}_h$, we can overestimate
$|f(x) - (P_h\bar{u}_h)(x) (P_h\bar{w})(x)|$
pointwise by setting
\[ \tilde{d}_f(x) \coloneqq \max\{|f(x) - u_\ell^i w_\ell|,|f(x) - u_\ell^i w_u|,|f(x) - u_u^i w_\ell|,|f(x) - u_u^i w_u|\} \]
for a.a.\ $x \in Q_i$ for all $Q_i \in \calQ_{h}$.
Thus we obtain the constant $c_{quad}(w_\ell, w_u, f, \hat{w}, j_0, u_{\ell,h},u_{u,h})$
by replacing $\|f - (P_h\bar{u}_h)(P_h\bar{w})\|_{L^1}$ by $\|\tilde{d}_f\|_{L^1}$
and $\TV(\bar{w})$ by $\frac{j(\hat{u}) + \alpha \TV(\hat{w}) - j_0}{\alpha}$
in \eqref{eq:C2} and multiplying the resulting constant by $\|\tilde{d}_u\|_H$.
\end{proof}
\begin{remark}
We note that the estimate $L_u \le \|\tilde{d}_u\|_{H}$ can be sharpened by using a posteriori information from the
approximate McCormick relaxations in order to bootstrap improved pointwise bounds $u_\ell$ and $u_u$ on $u$.
\end{remark}

\subsection{Properties of the control-to-state operator for (\ref{eq:pde_h})}\label{sec:control_to_state_operator}
We denote the nonlinear control-to-state operator that maps $w \in C$ to the solution
of the PDE \eqref{eq:pde_h} with the locally averaged state in the bilinearity as
$S : C \to U$. It is well defined and uniformly bounded due to the considerations
in \cref{sec:pdesetting,sec:verification_of_assumptions}. We briefly state Lipschitz continuity
and differentiability properties of the control-to-state operator so that we
verify the assumptions in \cite{leyffer2022sequential}. Consequently, the algorithmic
strategy from \cite{leyffer2022sequential,manns2023on} can be used to obtain stationary 
feasible (primal points) when additional integrality restrictions are
imposed on $w$ in \eqref{eq:ocp} while the analysis carried out above
allows us to obtain valid (approximate)
lower bounds. While the arguments are standard by following, for example, the
considerations on control-to-state operators in \cite{troltzsch2010optimal}, we provide
these proofs here because we were not able to find good references for the specific
Lipschitz, embedding, ...\ constants of the example in this article. We believe, however, that
they help  make informed assessments of the possible quality and performance of
our algorithmic approaches, although some of them may not be tight.
\begin{lemma}\label{lem:SLipschitzL1}
$S$ is Lipschitz continuous on the set $C$ as a function $L^1(0,1) \to U$.
\end{lemma}
\begin{proof}
Let $w_1$, $w_2 \in C$. Let $u_1 = S(w_1)$, $u_2 = S(w_2)$. We subtract the
weak two formulations, insert a suitable zero, and deduce
by means of H\"{o}lder's inequality the triangle inequality,
\eqref{eq:H10embedL2},
\eqref{eq:H10embedLinf}, and the nonexpansiveness of $P_h$:
\begin{align*}
\|\nabla(u_1 - u_2)\|_{H}^2
&= \left( (P_h w_1) (P_h u_1) - (P_h w_2)(P_hu_1), u_1 - u_2\right)_{H}\\
&\hphantom{=} + \left( (P_h w_2)(P_h u_1) - (P_h w_2)(P_h u_2), u_1 - u_2\right)_{H}\\
&\le \|u_1 - u_2\|_{L^\infty}\|u_1\|_{L^\infty}\|w_1 - w_2\|_{L^1} 
+ \max\{|w_\ell|,|w_u|\}\|u_1 - u_2\|_{H}^2 \\
&\le \left(
\frac{\|f\|_H}{4 c_2(w_\ell,w_u)}\|\nabla(u_1 - u_2)\|_{H}\|w_1 - w_2\|_{L^1} 
+ 
\frac{\max\{|w_\ell|,|w_u|\}}{\pi^2}\|\nabla(u_1 - u_2)\|_{H}^2 \right).
\end{align*}
An equivalent reformulation of this estimate gives
\begin{gather}\label{eq:LipschitzS}
 \|\nabla(u_1 - u_2)\|_{H} \le
\underbrace{\frac{\|f\|_{H} }{4 c_2(w_\ell,w_u)^2}}_{\eqqcolon L_S} \|w_1 - w_2\|_{L^1}.
\end{gather}
\end{proof}
Let $u$ solve \eqref{eq:pde_h} for $w \in C$.
For a given direction $s \in L^1(0,1)$ so that $w + s \in C$ and with $u = S(w)$,
we denote the (weak) solution to
\[ \int_0^1 \nabla q \nabla v + \int_0^1 (P_h q)(P_h w) v = -\int_0^1 (P_h u) (P_h s) v
   \quad \text{ for all } v \in U
\]
by $q$. We obtain that $q \in U$ is uniquely defined 
because $s \in C - w$ implies $L^\infty(0,1)$-bounds on $s$,
namely, $w_\ell - w_u \le s \le w_u - w_\ell$.
This gives $\|(P_hs)(P_hu)\|_H \le \|w_u - w_\ell|\|u\|_H$ and
a similar analysis as for \eqref{eq:pde_h} applies.
\begin{lemma}\label{lem:1st_der}
The operator $S : C \to U$ is Fr\'{e}chet differentiable with
respect to $L^1(0,1)$ in $C$. Let $q$ be as above for a given $s$
such that $w + s \in C$. Then $S'(w)s = q$.
For a fixed $s \in C$, the mapping $w \mapsto S'(w)s$ is
Lipschitz continuous in $C$ with respect to $L^1(0,1)$.
\end{lemma}
\begin{proof}
Clearly, the mapping $s \mapsto q$ is a bounded linear operator. We need to show
\[
\lim_{\|s\|_{L^1} \underset{s \in C}\to 0} \frac{\|S(w + s) + S(w) - q\|_{U}}{\|s\|_{L^1}} \to 0.
\]
Let $u_s = S(w + s)$, $u = S(w)$. Let $d \coloneqq u_s - u - q$. Then, we deduce
with the considerations above and \cref{lem:SLipschitzL1} and the Lipschitz constant
$L_S$ from \eqref{eq:LipschitzS} that
\begin{align*}
\|\nabla d\|_H^2 &= \big(P_h(u - u_s)(P_hs), d\big)_H - \big((P_hd)(P_hw),d\big)_H \\
&\le \frac{L_S}{2}\|\nabla d\|_{H}\|s\|_{L^1}\|s\|_{L^1}
+ \frac{\max\{|w_\ell|,|w_u|\}}{\pi^2}\|\nabla d\|_{H}^2,
\end{align*}
which yields
\[ \|\nabla d\|_H \le \frac{L_S}{2 c_2(w_\ell,w_u)}\|s\|_{L^1}^2.
\]
Since $\|\cdot\|_U$ is equivalent to $\|\nabla \cdot\|_H$ with the 
embedding \eqref{eq:H10embedL2}, this proves the claim.

For the Lipschitz continuity, let $w_1$, $w_2 \in C$,
$u_1 = S(w_1)$, $u_2= S(w_2)$, and $q_1 = S'(w_1)s$, $q = S'(w_2)s$.
Then, we obtain with $r = q_1 - q_2$ that
\begin{align*}
\|\nabla r\|_H^2
&= \big(P_h(u_2 - u_1)(P_hs),r\big)_H
 - \big((P_hr)(P_hw), r\big)_{H}
 - \big((P_hq_2) P_h(w_2 - w_1), r\big)_H\\
&\le \frac{L_S}{2} \|w_1 - w_2\|_{L^1}\|s\|_{L^1} \|\nabla r\|_H
+ \frac{\max\{|w_\ell|,|w_u|\}}{\pi^2}\|\nabla r\|_{H}^2
+ \frac{1}{2}\|\nabla q_2\|_{H}\|w_1 - w_2\|_{L^1}
\|\nabla r\|_{H}.
\end{align*}
With the arguments from the preceding subsections, we obtain
\begin{align*}
\|\nabla q_2\|_{H} &\le \frac{1}{c_2(w_\ell,w_u)}\|s\|_{L^\infty}\|u_2\|_H
\le \frac{|w_u - w_\ell|}{\pi c_2(w_\ell,w_u)^2}\|f\|_H ,
\end{align*}
so that we obtain the estimate
\[
\|\nabla r\|_H 
\le \underbrace{\frac{|w_u - w_\ell|}{2c_2(w_\ell,w_u)}
\left(
L_S 
+ 
\frac{\|f\|_H}{\pi c_2(w_\ell,w_u)^2}
\right)}_{\eqqcolon L_{S'}}
\|w_1 - w_2\|_{L^1}. \]
\end{proof}
For given directions $\phi$, $\psi \in L^1(0,1)$ so that $w + \phi$,
$w + \psi \in C$, we denote the (weak) solution to
\begin{multline}
\int_0^1 \nabla \xi \nabla v + \int_0^1 (P_h\xi)(P_h w) v =\\
- \int_0^1 P_h(S'(w)\phi) (P_h\psi) v
- \int_0^1 P_h(S'(w)\psi) (P_h\phi) v
\quad \text{ for all } v \in U\label{eq:2nd_der} 
\end{multline}
by $\xi$. As above, we obtain that $\xi \in U$ is uniquely defined
because $\phi$, $\psi \in C - w$ imply
$L^\infty(0,1)$-bounds on $\phi$, $\psi$, that is, $w_\ell - w_u \le \phi \le w_u - w_\ell$ and $w_\ell - w_u \le \psi \le w_u - w_\ell$.
This implies bounds on the lower-order order terms and a similar
a similar analysis as for \eqref{eq:pde_h} applies.

\begin{proposition}
Let $\phi \in C - w$ be given. The operator $w \mapsto S'(w)\phi$ is 
continuously Fr\'{e}chet differentiable with respect to $L^1(0,1)$ in $C$.
Its derivative in direction $\psi \in C - w$ is given by
$S''(w)[\psi,\phi] = \xi$, where $\xi$ is the solution to \eqref{eq:2nd_der}
above. Under these assumptions, it holds that
\[ \|S''(w)[\psi,\phi]\|_H \le \kappa \|\psi\|_{L^1}\|\phi\|_{L^1} \]
for some $\kappa > 0$.
\end{proposition}
\begin{proof}
For $\phi \in C - w$ and $\psi \in C - w$, let $q \coloneqq S'(w)\phi$,
$q_\psi \coloneqq S'(w + \psi)\phi$, $\tilde{q}_\psi \coloneqq S'(w)\psi$, 
$u_\psi = S(w + \psi)$, $u = S(w)$, $\xi$ be the solution to 
\eqref{eq:2nd_der}. Let $r \coloneqq q_\psi - q - \xi$. Then, we obtain
from H\"older's inequality and the embedding constants
\eqref{eq:H10embedL2}, \eqref{eq:H10embedLinf}
\begin{align*}
\|\nabla r\|_H^2 
&= \big(P_h(u + \tilde{q}_\psi - u_\psi)(P_h \phi),r\big)_H 
+ \big(P_h(q_\psi - q)(P_h \psi),r\big)_H
- \big((P_hr)(P_hw),r\big)_H\\
&\le \frac{1}{4}\|\nabla(u_\psi - u - \tilde{q}_\psi)\|_H \|\phi\|_{L^1} \|\nabla r\|_H
+ \frac{1}{4}\|\nabla(q_\psi - q)\|_H
 \|\psi\|_{L^1} \|\nabla r\|_H\\
&\quad + \frac{\max\{|w_\ell|,|w_u|\}}{\pi^2} \|\nabla r\|_H^2.
\end{align*}
Using the constants defined above and the estimates from the
proof of \cref{lem:1st_der}, we obtain
\[
4 c_2(w_\ell,w_u) \|\nabla r\|_H
\le \frac{L_S}{2 c_2(w_\ell,w_u)}\|\psi\|_{L^1}^2 \|\phi\|_{L^1}
+ L_{S'} \|\psi\|_{L^1}^2.
\]
Consequently, we obtain
\[
\lim_{\|\psi\|_{L^1} \underset{\psi \in C}\to 0} \frac{\|S(w + \psi) + S(w) - \xi\|_{U}}{\|\psi\|_{L^1}} \to 0,
\]
which proves the desired differentiability. 

Testing \eqref{eq:2nd_der} with $\xi = S''(w)[\psi,\phi]$ and using \eqref{eq:H10embedL2}
and similar computations as in the previous arguments, we obtain 
\begin{align*}
\|\nabla \xi\|_H^2 & \le \frac{\max\{|w_\ell|,|w_u|\}}{\pi^2}\|\nabla \xi\|_H^2
+\frac{\|f\|_H}{2 c_2(w_\ell,w_u)^2}\|\psi\|_{L^1}\|\phi\|_{L^1}\|\nabla \xi\|_H
\end{align*}
and in turn
\[ \|\xi\|_H \le \underbrace{\frac{\|f\|_H}{2\pi c_2(w_\ell,w_u)^3}}_{\eqqcolon \kappa}\|\psi\|_{L^1}\|\phi\|_{L^1}.
\]
The (Lipschitz) continuity of $w \mapsto S''(w)[\phi,\psi]$ can be shown with arguments similar to but more lengthy
than that for the Lipschitz continuity of the mapping $w \mapsto S'(w)s$ in \cref{lem:1st_der}.
\end{proof}

\section{Computational experiments for a 1D example}\label{sec:computational_experiments}
We start from the example given in \cref{sec:elliptic_example} and provide computational experiments that will serve
several purposes. We compare the approximate lower bounds obtained with approximate McCormick relaxations
before and after applying the OBBT procedure to a feasible primal point obtained with a gradient-based
optimization of \eqref{eq:ocp} and a primal point obtained with the SLIP algorithm from \cite{leyffer2022sequential,manns2023integer}
in the presence of an additional integrality restriction on $w$. We also compare the bounds with lower
bounds on \eqref{eq:ocp} that can be obtained with much less effort than the OBBT procedure.

We analyze how the bounds we obtain from \eqref{eq:mcchh} behave when the partitions
of the domain assumed in \cref{ass:mcch} \ref{itm:partitions} are uniformly refined,
thereby bringing the locally averaged McCormick relaxations closer to a pointwise limit.
Moreover, we apply the OBBT procedure \cref{alg:abstract_obbt} to assess whether and how much the
OBBT procedure tightens the bounds for a prescribed termination tolerance as well as to 
assess its computational effort.

This section is organized as follows. We describe our experiments
in \cref{sec:experiment_description}. In \cref{sec:discretization}
we describe how we approximate \eqref{eq:pde} and \eqref{eq:pde_h}
with a very fine finite-element discretization and use the resulting
discretized equations as state equations for \eqref{eq:ocp}, \eqref{eq:mcch},
and \eqref{eq:mcchh}.
We then give details on the practical implementation of the OBBT
procedure \cref{alg:abstract_obbt} in \cref{sec:obbt_implementation}.
The results are provided in \cref{sec:results}.

\subsection{Experiment description}\label{sec:experiment_description}
We consider the instance of \eqref{eq:ocp} from \cref{sec:elliptic_example}
with the following choices of the parameters and fixed inputs.
The source term $f$ of the PDEs \eqref{eq:pde} and \eqref{eq:pde_h}
is chosen as $f(x) \coloneqq 6$ for all $x \in [0,1]$.
Regarding the bounds, we set $w_\ell \coloneqq -4$, $w_u \coloneqq 4$. For the penalty
parameter, we choose $\alpha \coloneqq 2.5 \cdot 10^{-4}$. For the function $u_d$
in the tracking-type objective (see  \eqref{eq:juw}, we choose
\begin{align*}
u_d(x) &\coloneqq 3
\Big( 1.5 x (1 - x) \chi_{[0,0.25]}(x) 
+ 1.5 x (1 - x) \chi_{[0.75, 1]}\\
&\hphantom{\coloneqq} 
+ (0.28125 + 3(x - 0.25)) \chi_{(0.25,0.4]}(x) \\
&\hphantom{\coloneqq} 
+ (0.73125 - 3(x - 0.6)) \chi_{[0.6,0.75)}(x) 
+ 2\chi_{(0.4,0.6)}(x)
\Big).
\end{align*}

Our main object of interest is to assess the quality of the approximate 
McCormick relaxations, the effect of the OBBT procedure \cref{alg:abstract_obbt}.
To this end, we perform the following computations.
\begin{itemize}
\item We execute a local gradient-based NLP solver (L-BFGS-B) using Scipy's implementation 
      \cite{2020SciPy-NMeth} in order to obtain a stationary point for \eqref{eq:ocp}
      and thus a low upper bound on $m_{\eqref{eq:ocp}}$.
      In order to be able to apply this method to the nondifferentiable total variation
      seminorm, the latter is smoothed by using an overestimating Huber regularization for the absolute
      value in the integrand with smoothing parameter $10^{-3}$.
\item We add the additional integrality constraint $w(x) \in \Z$ and execute the SLIP 
      algorithm \cite{leyffer2022sequential,manns2023on} using the subproblem solver
      described in \cite{severitt2023efficient} in order to obtain a low upper bound
      for the integrality-constrained version of \eqref{eq:ocp}.
\item We use monotonicity properties of the PDE that we do not exploit in our computations
	  elsewhere in order to compute a pointwise McCormick envelope with bounds $u_\ell$, $u_u$
	  that are as tight as possible in order to assess the quality of the approximate
	  McCormick relaxations a posteriori.
\item We compute approximate McCormick relaxations, that is, solutions to \eqref{eq:mcchh},
      with and without bound tightening (and subsequent improvement of the involved constants)
      as well as the induced lower bounds on \eqref{eq:ocp} for decreasing values of $h$
      using the estimates from \cref{sec:apriori}.
\end{itemize}
All experiments were executed on a laptop computer with
Intel (TM) i7-11850H CPU clocked at 2.5\;GHz and 64\;GB main memory.

\subsection{Baseline PDE discretization with Ritz--Galerkin ansatz}\label{sec:discretization}
We consider conforming finite elements and thus a finite-dimensional subspace
$U_N \subset U$ with dimension $N \in \N$. By means of the Lax--Milgram
lemma, we obtain that there exists a unique solution $u_N \in U_N$, the Ritz approximation, of the weak formulation
on $U_N$ of the PDEs \eqref{eq:pde}:
\begin{gather}\label{eq:pde_N}
\int_0^1 \nabla u_N \nabla v_N 
+ \int_0^1 w u_N v_N
= \int_0^1 f v_N
\quad \text{ for all } v_N \in U_N.
\end{gather}
and \eqref{eq:pde_h}
\begin{gather}\label{eq:pde_h_N}
\int_0^1 \nabla u_N \nabla v_N 
+ \int_0^1 (P_hw) (P_h u_N) v_N
= \int_0^1 f v_N
\quad \text{ for all } v_N \in U_N.
\end{gather}
The solutions to \eqref{eq:pde_N} satisfy
\[ \|\nabla u_N\|_H \le \frac{1}{c_1(w_\ell)}\|f\|_H,
\]
and the solutions to \eqref{eq:pde_h_N} satisfy
\[ \|\nabla u_N\|_H \le \frac{1}{c_2(w_\ell,w_u)}\|f\|_H.
\]
The whole analysis from \cref{sec:mccormick,sec:elliptic_example} and, in 
particular, all of the estimates derived for our example PDE in \cref{sec:elliptic_example}
still hold when discretizing the variational form of the PDE with a conforming finite-element
setting and using a piecewise constant control ansatz for $w$ on the same or a coarser grid. The only
required change is to insert $U_N$ as the state space for $U$. We use first-order Lagrange
elements and a discretization of $(0,1)$ into $N = 2,048$ intervals.
Regarding the additional control variable $z$ in the state equation $Au + z = f$
of the pointwise envelope in \eqref{eq:mcc}, we also choose a piecewise constant ansatz on the
$2,048$ intervals, which introduces a small approximation error because a piecewise constant function
cannot completely capture the product of a first-order Lagrange element and a piecewise constant function.

Consequently, in our implementation of \eqref{eq:ocp} and \eqref{eq:mcc}, we use following discretization
that serves as a baseline and substitutes the infinite-dimensional setting in our experiments:
\begin{itemize}
\item discretization of $u$ with first-order Lagrange elements on $N = 2,048$ intervals,
\item discretization of $w$ with piecewise constant functions on $N = 2,048$ intervals, and
\item discretization of $z$ with piecewise constant functions on $N = 2,048$ intervals.
\end{itemize}

\begin{remark}\label{rem:branch_and_bound}
In this \emph{discretize-then-optimize} setting, it is possible to directly employ
a branch-and-bound procedure on the fully discretized version of \eqref{eq:ocp} when
additional integrality restrictions are imposed on $w$ such as $w(x) \in \Z$. Then
\eqref{eq:pde} is just replaced by \eqref{eq:pde_N}; the lower bounds are computed by
means of McCormick relaxations using \eqref{eq:pde_h_N}, and the upper bounds  are, for example,
computed by means of the SLIP algorithm  \cite{leyffer2022sequential,manns2023integer}.
\end{remark}

\begin{remark}
Since the problem considered in this section is defined on an interval,
that is in one dimension, it is possible to reformulate it as an optimal control problem
with a multiple point boundary-value problem as constraint and apply optimal control techniques
like single- or multiple-shooting or collocation for discretization and solution, see the articles
\cite{betts1998survey,rao2009survey} for overviews over such techniques.
Due to our focus on applicability to PDEs, we have decided to use a finite-element
discretization in this work.
\end{remark}

\subsection{Practical implementation of OBBT / \cref{alg:abstract_obbt}}\label{sec:obbt_implementation}
In our implementation of \cref{alg:abstract_obbt}, we first minimize the lower bounds $u_\ell^i$ one by one
along the order of the intervals $i = 1,\ldots,N_h$. Then we maximize the upper bounds $u_u^i$
one by one along the order of the intervals $i = 1,\ldots,N_h$. Then we repeat this process.
We terminate when a complete run of minimization of the $u_\ell^i$ (or maximization of the $u_u^i$)
does not produce a tightening of any of the bounds by more than a prescribed tolerance of $10^{-6}$.

Moreover, in order to avoid numerical problems and in turn incorrect results,
additional safeguarding was necessary. Specifically, when bounds are
set to the computed value in \cref{alg:abstract_obbt} ln.\ \ref{ln:set_new_lb}
or ln.\ \ref{ln:set_new_ub}, the value might be slightly too sharp
because of the numerical precision of the subproblem solver for the OBBT problems.
This, however, can cause the lower bounds to increase and the upper bounds
to decrease to incorrect values in further tightenings and the feasible
set even contracting to an empty set (when the lower and upper bounds cross
each other). In order to avoid this situation, the value $10^{-7}$ is added to the
upper bounds and subtracted from the lower bounds computed by the subproblem
solver for the OBBT problems before they are assigned as new bounds
and \cref{alg:abstract_obbt}  moves on to tighten the next bound.

Because \cref{alg:abstract_obbt} requires valid initial bounds, we use
the bounds that can be inferred from \eqref{eq:Linfbound} and our choice of
$f$. Specifically, we use the initial bounds $u_\ell^i = -5.0444$ and 
$u_u^i = 5.0444$ for all $i \in \{1,\ldots,N_h\}$. 

\subsection{Results}\label{sec:results}
We first ran Scipy's implementation \cite{2020SciPy-NMeth} of
L-BFGS-B on \eqref{eq:ocp}, where we used the unique solvability of 
the (discretized) PDE constraint to integrate it into the
objective and used adjoint calculus in order to compute the derivative. 
To use this gradient-based solver, we have smoothed the
total variation seminorm using a Huber regularization with smoothing 
parameter $10^{-3}$. We inserted the solution into the nonsmooth
objective and obtained an upper bound of \num{8.3808e-02}
on the optimal solution to \eqref{eq:ocp}.

Then we added the integrality constraint $w(x) \in \Z$ to \eqref{eq:ocp}
and executed the SLIP algorithm \cite{leyffer2022sequential,manns2023on} 
with the subproblem solver from \cite{severitt2023efficient} to compute
an upper bound on the optimal solution to \eqref{eq:ocp} with the
additional integrality constraint $w(x) \in \Z$. The resulting
upper bound was \num{8.5551e-02}.

Regarding the lower bounds, we first compute pointwise lower bounds
for comparison. For our computational example, we compute the optimal 
objective value for the uniform bounds $u_\ell \equiv -5.0444$ and
$u_u \equiv 5.0444$ that are enforced on the nodes of our first-order
Lagrange elements. We first compute a lower bound by omitting the
total variation term from the objective ($\alpha = 0$),
which also gives us an overall lower bound on the tracking-type term
in the objective. The resulting linearly constrained convex
quadratic program is solved by using Gurobi \cite{gurobi}, and
the computed lower bound is \num{6.7701e-02}.
Then we consider \eqref{eq:mcc} with the same bounds but including
the total variation term. Again, the resulting linearly constrained convex
quadratic program is solved by using Gurobi \cite{gurobi},
and the computed lower bound  is \num{6.8649e-02}.
Moreover, we can characterize the tightest possible pointwise McCormick 
relaxation exactly. Specifically, the monotonicity properties of the 
Laplacian yield that the discretization of \eqref{eq:pde} attains its 
pointwise minimum  $u_{\min}$ for the choice $w(x) = w_u = 4$ for all
$x \in \Omega$ and its pointwise maximum $u_{\max}$
for the choice $w(x) = w_{\ell} = -4$ for all $x \in (0,1)$. 
Consequently, we execute \eqref{eq:mcc} with the bounds $u_\ell = u_{\min}$
and $u_u = u_{\max}$ to obtain the tightest lower bound the McCormick relaxation can possibly achieve.
We obtain a higher optimal objective value of \num{8.3679e-02}.

\sisetup{scientific-notation=true,round-mode=places,round-precision=3}
This means that if this perfect information on the bounds $u_u$
and $u_\ell$ is available when solving \eqref{eq:mcc}, the gap
between upper and lower bounds gets reduced from \num{0.016902}
to \num{0.0018719999999999893}
in the presence of the constraint $w(x) \in \Z$ (MINLP case) and
from \num{0.015159000000000011} to \num{0.00012900000000000135}
in the continuous case. In both cases, the relative gap, which is 
generally computed as the difference between upper and lower bound
divided by the lower bound,
gets reduced by an order of magnitude. These results are tabulated
in \cref{tbl:exact_bounds}.

\sisetup{scientific-notation=true,round-mode=places,round-precision=3}
\begin{table}[t!]
	\caption{Exact upper and lower bounds and relative gaps (ratio of the difference between upper and lower bound to the lower bound)
	for \eqref{eq:ocp} and
	its counterpart with integrality restriction $w(x) \in \Z$.}\label{tbl:exact_bounds}
	\centering
	\begin{adjustbox}{width=\textwidth}	
		\begin{tabular}{r|ccccc}
			\toprule
			 & \multicolumn{2}{c}{Upper bounds}
			 & \multicolumn{3}{c}{Lower bounds}
			 \\
			 & SLIP ($w(x) \in \Z$)
			 & L-BFGS-B
			 & \eqref{eq:mcc} (tightest) 
			 & \eqref{eq:mcc} 
			 & \eqref{eq:mcc} ($\alpha = 0$)
			 \\
			 \midrule
			 Value 
			 & \num{8.5551e-02}
			 & \num{8.3808e-02}
			 & \num{8.3679e-02}
			 & \num{6.8649e-02} 
			 & \num{6.7701e-02}
			 \\
			 Rel.\ gap (MINLP) 
			 &
			 &
			 & \num{0.02237120424479247}
			 & \num{0.24620897609579165}
			 & \num{0.26365932556387656}
			 \\
			 Rel.\ gap (NLP) 
			&
			&
			& \num{0.0015416054207148906}
			& \num{0.22081894856443665}
			& \num{0.23791376789116847}
			\\		 
			\bottomrule
		\end{tabular}
	\end{adjustbox}
\end{table}
In the remaining experiments, we use these results as a baseline
to compare them with the (approximate) lower bounds obtained
by solving instances of \eqref{eq:mcchh}. We executed our
experiments on \eqref{eq:mcchh} on uniform partitions of the domain
$(0,1)$ into $N_h \in \{8,16,32,64,128,256,512,1024\}$ intervals
with the corresponding values $h = 2^{-N_h}$.

We now consider the approximate lower bounds that are obtained by solving
\eqref{eq:mcchh} for the aforementioned values of $h$.
In particular, we assess their approximation
quality and the running times that are required to compute them
by means of the OBBT procedure for the different values of $h$.
To obtain a valid lower bound on
\eqref{eq:ocp}, we need to subtract $c_{quad} h^2$ from the 
optimal objective value of \eqref{eq:mcchh}; see \eqref{eq:quad_bound}.
In practice, the values $j_0$, $\|\tilde{d}_u\|_H$, and 
$\|\tilde{d}_f\|_H$ are not known precisely because this would imply
that tight pointwise McCormick relaxations have already been computed.

We therefore use two constants in our considerations that allow us to present a range where we expect that the a priori estimates
lie when all of the other constants are known. The
larger and thus more conservative constant is obtained by
choosing $j_0 = 0$ and computing $\tilde{d}_u$ and $\tilde{d}_f$
with the conservative bounds $u_\ell \equiv -5.0444$ and 
$u_u \equiv 5.0444$. The smaller constant is computed by choosing
$j_0 = \num{6.7701e-02}$
as well as the optimal bounds $u_\ell = u_{\min}$ and $u_u = u_{\max}$,
which are of course not available in a realistic setting.
In our setting, the conservative bound is almost $50$ times
larger than the tight one. Both values are given in \cref{tbl:cquad_value}.
\sisetup{scientific-notation=true,round-mode=places,round-precision=3}
\begin{table}
	\centering
	\caption{Values for $c_{quad}$ in \eqref{eq:quad_bound}.}
	\label{tbl:cquad_value}
		\begin{tabular}{cc}
			\toprule
			conservative & tight \\
			\midrule
			\num{56026.22410496081} & \num{1131.734721750359}\\
			\bottomrule
		\end{tabular}	 
\end{table}

For our example, the initial bounds that can be deduced from \eqref{eq:Linfbound}
are $u_\ell^i = -5.0444$ and $u_u^i = 5.0444$ for all $i \in \{1,\ldots,N_h\}$.
After executing our implementation of \cref{alg:abstract_obbt}, the lower bounds and upper
bounds are much closer to each other and reflect the structure of the possible PDE solutions;
that is, they  tend to zero toward the boundaries
of the computational domain. For the finest computed case with $h = 2^{-10}$, $N_h = 1024$, 
and a termination tolerance of $10^{-6}$ for our implementation of the OBBT procedure,
the lower bounds $u_\ell^i$ for $P_hu$ vary between $0$ and $\num{5.2795e-01}$ and the upper bounds
$u_u^i$ vary between $0$ and $\num{1.2762e+00}$, thereby giving a much
tighter envelope in which $P_h u$ for the solutions to \eqref{eq:pde_h} can lie.
Moreover, the optimal solution computed for \eqref{eq:mcchh}
is  close to the optimal solution computed for \eqref{eq:mcc}
for fine values of $h$. Both findings are visualized in  \cref{fig:mcc_and_obbt_envelope}.
\begin{figure}[t]
\begin{center}
	\begin{subfigure}{0.49\textwidth}
		\includegraphics[width=\textwidth]{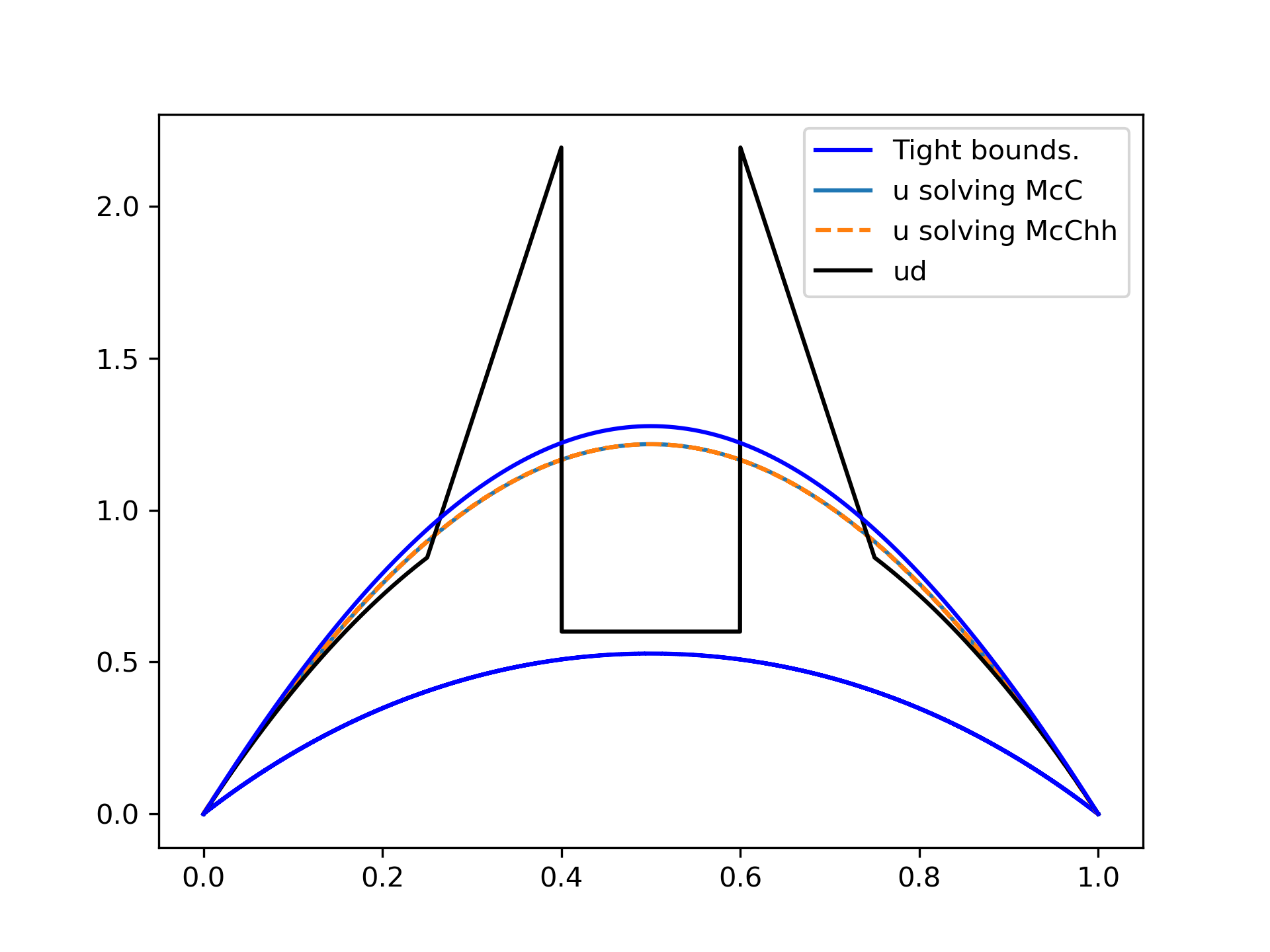}
		\caption{Solutions to \eqref{eq:mcc} and \eqref{eq:mcchh}.}
	\end{subfigure}
	\hfill
	\begin{subfigure}{0.49\textwidth}
		\includegraphics[width=\textwidth]{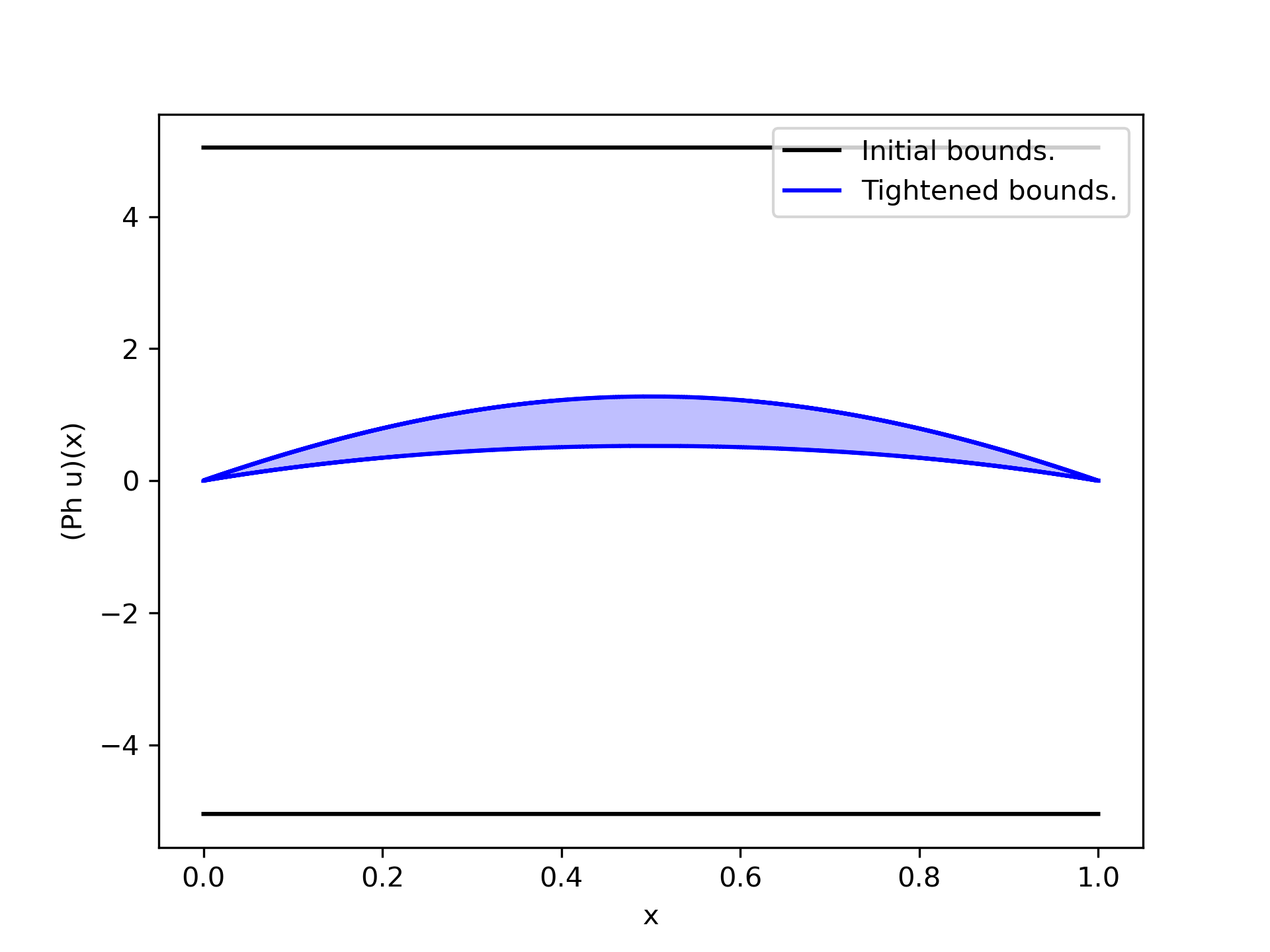}
		\caption{PDE solution envelope after OBBT.}
	\end{subfigure}
\end{center}
\caption{(a) Solutions to \eqref{eq:mcc} (cyan, solid) and \eqref{eq:mcchh} for $h = 2^{-10}$,
	(orange, dashed) with $u_{\min}$, $u_{\max}$
	(dark blue, solid), and tracking function $u_d$
	(black, solid).
(b) Initial (black) and tightened (blue) lower and upper bounds on $P_h u$ for solutions $u$ to \eqref{eq:pde_h} for 
$h = 2^{-10}$. The area between the tightened lower and upper bounds, where the solution
of the PDE can attain values, is colored
with a less intense blue.}\label{fig:mcc_and_obbt_envelope}
\end{figure}

\sisetup{scientific-notation=true,round-mode=places,round-precision=3}
\begin{table}[t]
	\caption{Running times (in seconds) for the OBBT algorithm
	for different values of $h$
	on instances of \eqref{eq:mcchh}
	and \eqref{eq:mcc} (mesh size $2^{-11}$).
	}\label{tbl:obbt_runtimes}
	\centering
	\begin{adjustbox}{width=\textwidth}	
		\begin{tabular}{r|ccccccccc}
			\toprule
			& \multicolumn{8}{c}{\eqref{eq:mcchh} with $h = $} &
			\eqref{eq:mcc} \\
			& $2^{-3}$ & $2^{-4}$ & $2^{-5}$ & $2^{-6}$ & $2^{-7}$ & $2^{-8}$ & $2^{-9}$ & $2^{-10}$ &  \\
			\midrule 
			Time [s]
			& \num{8.21241403398744}
			& \num{16.1153287569905}
			& \num{31.7425664890034}
			& \num{70.25973947100281}
			& \num{165.219088275}
			& \num{456.513897383003}
			& \num{1594.48611089699}
			& \num{10795.856385985}
			& \num{2.1418e+04} \\
			\bottomrule
		\end{tabular}
	\end{adjustbox}
\end{table}
Generally, the increase in compute times for the bounds
is worse than linear with decreasing values of $h$. 
The runtime for the OBBT procedure is $\num{8.21241403398744}$ seconds
for $h = 2^{-3}$ and increases to $\num{10795.856385985}$
seconds
for $h = 2^{-10}$ for \eqref{eq:mcchh} compared with a pointwise bound
computation with $\num{2.1418e+04}$ seconds (mesh size
$2^{-11}$) for \eqref{eq:mcc}.
All running times of the OBBT algorithm are tabulated in \cref{tbl:obbt_runtimes}.

Next, we assess the approximation quality of the bounds
induced by the solutions to \eqref{eq:mcchh}
with and without bound tightening.
In both cases we observe that the infima converge
numerically to their pointwise counterparts of \eqref{eq:mcc} that are tabulated
in \cref{tbl:exact_bounds}. When subtracting 
$c_{quad} h^2$ from each of these values for the conservative
and tight choices of $c_{quad}$ from \cref{tbl:cquad_value},
we obtain valid lower bounds and observe that only for
the smallest grid with $h= 2^{-10}$  all of the lower bounds
 are positive and thus beat the trivial lower bound $0$
that can be found by inspecting the objective.
If the OBBT procedure is applied and a tight choice of $c_{quad}$
is used, the lower bound is positive for the choices of
$h \le 2^{-7}$, and the computed bound beats the pointwise
bound without OBBT for $h = 2^{-9}$ and $h = 2^{-10}$, thereby
yielding the second and third best lower bounds of all computed lower bounds.
All computed values are given in \cref{tbl:mcchh_optimal_values_obbt}.
\sisetup{scientific-notation=true,round-mode=places,round-precision=3,explicit-sign=+}
\newcommand{\highlightcell}[1]{{\cellcolor[gray]{0.85}} #1}
\begin{table}[ht]
	\centering
	\caption{Optimal objective values achieved for \eqref{eq:mcchh} and 
		induced bounds for \eqref{eq:ocp} using the estimate 
		\eqref{eq:quad_bound} with and without OBBT with 
		conservative (c) and tight (t) estimate  on $c_{quad}$. The bounds that are able to beat the trivial bound $0$ that can be directly inferred from
		the nonnegative objective terms are highlighted in gray.}
	\label{tbl:mcchh_optimal_values_obbt}
	\begin{adjustbox}{width=\textwidth}
		\begin{tabular}{lr|llllllll}
			\toprule
			OBBT & $h$
			& $2^{-3}$ & $2^{-4}$ & $2^{-5}$& $2^{-6}$ & $2^{-7}$ & $2^{-8}$ & $2^{-9}$& $2^{-10}$\\
			\midrule
			No & $m_{\eqref{eq:mcchh}}$
			& \num{7.1532e-02}
			& \num{7.0280e-02}
			& \num{6.8681e-02}
			& \num{6.8670e-02}
			& \num{6.8665e-02}
			& \num{6.8656e-02}
			& \num{6.8649e-02}
			& \num{6.8649e-02} \\
			Yes & $m_{\eqref{eq:mcchh}}$
			& \num{8.3730e-02}
			& \num{8.3702e-02}
			& \num{8.3681e-02}
			& \num{8.3680e-02}
			& \num{8.3679e-02}
			& \num{8.3679e-02}
			& \num{8.3679e-02} 
			& \num{8.3679e-02} \\
			No & LB \eqref{eq:quad_bound} (c) &
			\num{-8.7534e+02} &
			\num{-2.1878e+02} &
			\num{-5.4644e+01} &
			\num{-1.3610e+01} &
			\num{-3.3509e+00} &
			\num{-7.8624e-01} &
			\num{-1.4507e-01} 
			& \highlightcell{\num{1.5219e-02}} \\
			Yes & LB \eqref{eq:quad_bound} (c)
			& \num{-8.7533e+02}
			& \num{-2.1877e+02}
			& \num{-5.4629e+01}
			& \num{-1.3595e+01}
			& \num{-3.3359e+00}
			& \num{-7.7121e-01}
			& \num{-1.3004e-01} 
			& \highlightcell{\num{3.0248e-02}} \\
			No & LB \eqref{eq:quad_bound} (t)
			& \num{-1.7612e+01}
			& \num{-4.3506e+00}
			& \num{-1.0365e+00}
			& \num{-2.0763e-01}
			& \num{-4.1057e-04}
			& \highlightcell{\num{5.1387e-02}}
			& \highlightcell{\num{6.4332e-02}}
			& \highlightcell{\num{6.7570e-02}} \\			
			Yes & LB \eqref{eq:quad_bound} (t)
			& \num{-1.7600e+01}
			& \num{-4.3371e+00}
			& \num{-1.0215e+00}
			& \num{-1.9262e-01}
			& \highlightcell{\num{1.4604e-02}}
			& \highlightcell{\num{6.6410e-02}}
			& \highlightcell{\num{7.9362e-02}}
			& \highlightcell{\num{8.2600e-02}} \\
			\bottomrule
		\end{tabular}	
	\end{adjustbox}
\end{table}

We assess the quality of the approximate bounds $m_{\eqref{eq:mcchh}}$
after optimization-based bound tightening with respect to $m_{\eqref{eq:mcc}}$.
We observe that the difference as well as the relative difference becomes
very small for small values of $h$, in particular several magnitudes
smaller than the error margin of the a priori estimate. The obtained values
are given in \cref{tbl:actual_difference}.
\sisetup{scientific-notation=true,round-mode=places,round-precision=3,explicit-sign=}
\begin{table}[ht]
	\centering
	\caption{Difference and relative difference between $m_{\eqref{eq:mcc}}$
		for $u_\ell = u_{\min}$, $u_u = u_{\max}$ and
		$m_{\eqref{eq:mcchh}}$ for
		$u_\ell$, $u_u$ computed using the OBBT
		procedure.}
	\label{tbl:actual_difference}
	\begin{adjustbox}{width=\textwidth}
		\begin{tabular}{r|llllllll}
			\toprule
			h
			& $2^{-3}$ & $2^{-4}$ & $2^{-5}$& $2^{-6}$ & $2^{-7}$ & $2^{-8}$ & $2^{-9}$& $2^{-10}$\\
			\midrule
			$|m_{\eqref{eq:mcchh}} - m_{\eqref{eq:mcc}}|$
			& \num{5.1058e-05}
			& \num{2.2474e-05}
			& \num{2.2915e-06}
			& \num{5.7399e-07}
			& \num{1.4499e-07}
			& \num{3.8199e-08}
			& \num{1.1759e-08}
			& \num{3.4362e-08} \\
			$\frac{|m_{\eqref{eq:mcchh}} - m_{\eqref{eq:mcc}}|}{m_{\eqref{eq:mcc}}}$
			& \num{6.1016e-04}
			& \num{2.6858e-04}
			& \num{2.7384e-05}
			& \num{6.8594e-06}
			& \num{1.7327e-06}
			& \num{4.5649e-07}
			& \num{1.4053e-07}
			& \num{4.1064e-07} \\			
			\bottomrule
		\end{tabular}
	\end{adjustbox}
\end{table}

\section{Computational experiments for an example in 2D}\label{sec:computational_experiments_2D}
While an analysis that verifies all of our assumptions for a PDE on a
multi-dimensional domain is beyond the scope of this article, we still
present an example for which we are certain that we can verify them by means
of the techniques mentioned in \cref{sec:mccormick_pointwise}. Moreover,
the techniques from \cref{sec:apriori} can be combined with
regularity estimates from \cite{grisvard2011elliptic} to derive
a priori estimates as well.
Specifically, we consider a convection-diffusion boundary
value problem that leans on the setting in \cite{baraldi2024domain}.

This section is organized as follows. We describe our experiments
in \cref{sec:2d_experiment_description}. In \cref{sec:2d_discretization}
we describe how we approximate \eqref{eq:pde2d} and its counterpart with
$u$ and $w$ replaced by $P_h u$ and $P_h w$ 
with a finite-element discretization and use the resulting discretized equations
as state equations for \eqref{eq:ocp}, \eqref{eq:mcch}, and \eqref{eq:mcchh}.
We give details on the implementation of the OBBT procedure
in \cref{sec:obbt_implementation}. The results are provided in \cref{sec:results}.

\subsection{Experiment description}\label{sec:2d_experiment_description}
Regarding the state equation, we set $\Omega = (0,1)^2$ and $W = \{0,1\}$.
Let $w$ with $w(x) \in  W$ a.e.\ be given. Then the state vector $u$ is given
by the solution to
\begin{gather}\label{eq:pde2d}
\begin{aligned}
-\varepsilon \Delta u + c_1 \cdot \nabla u + c_2 u w &= f \quad\text{in } \Omega \\
u&= 0 \quad\text{on } \{0,1\} \times (0,1) \cup ((0,0.25) \cup (0.75,1)) \times \{1\} \\
u&= \sin(2 \pi (x_1 - 0.25)) \quad\text{on } (0.25,0.75) \times \{1\} \\
\partial_n u &= 0 \quad\text{on } (0,1) \times \{0\},
\end{aligned}
\end{gather}
where $\varepsilon = 0.04$, $c_2 = 4$, $c_1(x) = (\begin{matrix} \sin(\pi x_1)
& \cos(2 \pi x_2)\end{matrix})^T$ for $x \in \Omega$, 
$f(x) = \sin(2 \pi x_1 + 2 \pi x_2) + 3$ for $x \in \Omega$.
Let $S$ be the control-to-state operator of \eqref{eq:pde2d}.
We choose the objective
\[ j(u) \coloneqq \frac{1}{2}\|u - u_d\|_{L^2}^2, \]
where we compute $u_d$ as follows. We replace the coefficient $c_1$
in \eqref{eq:pde2d} by $\tilde{c}_1(x) = (\begin{matrix} -x_2
& 2 x_1\end{matrix})^T$ and solve this modified boundary value problem
for the control $w = 2.5 \chi_{A} - 4(x_1 - 0.35)^3 \chi_B - 6(x_2 - 0.35)^3\chi_B$,
where we have $A = (0,0.35)^2$ and $B = \Omega\setminus (0,0.35)^2$.

Since our main object of interest is to assess the quality of the approximate 
McCormick relaxations and the effect of the OBBT procedure,
we perform the following computations, where $h$ denotes the mesh size.
\begin{itemize}
	\item We execute a local gradient-based NLP solver in order to obtain a
	stationary point for \eqref{eq:ocp}	and thus a low upper bound on $m_{\eqref{eq:ocp}}$.
	\item We add the additional integrality constraint $w(x) \in \Z$ and execute the SLIP algorithm \cite{leyffer2022sequential,manns2023integer} to obtain a low
	upper bound for the integrality-constrained version of \eqref{eq:ocp}.
	\item Similar to the 1D experiment, we compute a pointwise McCormick envelope
	with bounds $u_\ell$, $u_u$ using OBBT on the nodal basis of the finest
	discretization we have available. This then serves as a baseline in order to assess
	the quality of the approximate McCormick relaxations.
	\item We compute approximate McCormick relaxations, that is, solutions to \eqref{eq:mcchh}, with and without bound tightening for decreasing values of $h$,
	where $h$ is the mesh size of a uniform grid of squares.
\end{itemize}
We have carried out the experiments on a node of the Linux HPC cluster LiDO3 with two
AMD EPYC 7542 32-Core CPUs and 64 GB RAM.

\subsection{Baseline PDE discretization with Ritz--Galerkin ansatz}\label{sec:2d_discretization}
We decompose the domain into a $48 \times 48$ grid of squares, which consist
of four triangles each, on which the state vector of \eqref{eq:pde2d} is
defined. We consider conforming finite elements and thus a finite-dimensional subspace
$U_N \subset U$ with dimension $N \in \N$, specifically, we use first-order Lagrange
elements. The term pointwise envelope refers to the pointwise constraint
bounds are enforced on the nodes of the first-order Lagrange ansatz.
Regarding the additional control variable $z$ in the state equation
of the pointwise envelope in \eqref{eq:mcc}, we also choose a piecewise constant 
ansatz on all triangles similar to the one-dimensional case.

Consequently, in our implementation of \eqref{eq:ocp} and \eqref{eq:mcc}, we use
following discretization that serves as a baseline and substitutes the infinite-dimensional setting in our experiments:
\begin{itemize}
	\item discretization of $u$ with first-order Lagrange elements on $4 \times 48 \times 48$ triangles,
	\item discretization of $w$ with piecewise constant functions on $48 \times 48$ squares, and
	\item discretization of $z$ with piecewise constant functions on $4 \times 48 \times 48$ triangles.
\end{itemize}

\subsection{Practical implementation of OBBT}\label{sec:2d_obbt_implementation}
Our implementation \cref{alg:abstract_obbt} differs from the one described in 
\cref{sec:obbt_implementation} to be able to parallelize over bound computations
due to the very long compute times.
We initialize the lower bounds $u_\ell$ both for the pointwise and the
locally averaged averaged McCormick envelopes by $-10^3$ and the upper bounds
$u_u$ by $10^3$, which are high enough for this example to safely assume that
the true bounds lie within them.

For a given set of bounds, we minimize the lower bounds $u_\ell^i$ and maximize the 
upper bounds $u_u^i$ for all nodes/grid cells in parallel. Then we use this set of
bounds as the new set of bounds and compute another round of tightened bounds.
We repeat this seven times so that eight rounds of bound tightening are executed in total.
Again, in order to avoid numerical problems and obtain correct results, we applied the
additional safeguarding from  \cref{sec:obbt_implementation} but note that a coarser
safety tolerance of $10^{-2}$ was necessary here and we needed to set the
parameter \emph{BarHomogeneous} to one in Gurobi to  prevent incorrect identification
of infeasibility in a few cases that occurred in the eighth round of OBBT on the finest
discretization.

\sisetup{scientific-notation=true,round-mode=places,round-precision=4}
\subsection{Results}\label{sec:2d_results}
We first ran the local gradient-based NLP solver, in which we
employ an anisotropic discretization of the total variation seminorm based
on the piecewise constant control function ansatz, see also Appendix B in 
\cite{manns2023integer} and section 2 in \cite{manns2024discrete}. In the
trust-region method, we use a linear model for the first part of the objective
and handle the trust-region seminorm as in \cite{manns2023integer}
so that our trust-region subproblems are linear programs after discretization.
The resulting objective value was \num{3.2337}.

Then we added the integrality constraint $w(x) \in \Z$ to \eqref{eq:ocp}
and executed the SLIP algorithm \cite{leyffer2022sequential,manns2023integer}
to compute an upper bound on the optimal solution to \eqref{eq:ocp} with the
additional integrality constraint $w(x) \in \Z$. The resulting
upper bound was \num{3.2340}.

Regarding the lower bounds, we first compute pointwise lower bounds
using OBBT as described above. Then we solve
\eqref{eq:mcc} with these bounds using Gurobi \cite{gurobi}
as well as with the initial bounds ($-10^3$ for lower and $10^3$ for upper bounds).
The computed lower bound for the initial bounds is $\num{3.025491e-11}$ and
the computed lower bound for the tightened bounds \num{3.1347}.
These results are tabulated in \cref{tbl:2d_baseline_bounds}.

\sisetup{scientific-notation=true,round-mode=places,round-precision=4}
\begin{table}[t!]
	\caption{Upper and lower bounds and relative gaps (ratio of the difference between upper and lower bound to the lower bound)
		for \eqref{eq:ocp} and
		its counterpart with integrality restriction $w(x) \in \Z$.}\label{tbl:2d_baseline_bounds}
	\centering
	\begin{adjustbox}{width=\textwidth}	
		\begin{tabular}{r|cccc}
			\toprule
			& \multicolumn{2}{c}{Upper bounds}
			& \multicolumn{2}{c}{Lower bounds}
			\\
			& SLIP ($w(x) \in \Z$)
			& NLP solver
			& \eqref{eq:mcc} (OBBT)
			& \eqref{eq:mcc} (initial)
			\\
			\midrule
			Value 
			& \num{3.2340}
			& \num{3.2337}
			& \num{3.1347}
			& \num{3.025491e-11}
			\\
			Rel.\ gap (MINLP) 
			&
			&
			& \num{0.0316894758669090467}
			& \num{106892963818.75685}
			\\
			Rel.\ gap (NLP) 
			&
			&
			& \num{0.031595686987590514}
			& \num{106883246387.76796}
			\\		 
			\bottomrule
		\end{tabular}
	\end{adjustbox}
\end{table}
In the remaining experiments, we use these results as a baseline
to compare them with the (approximate) lower bounds obtained
by solving instances of \eqref{eq:mcchh}. We executed our
experiments on \eqref{eq:mcchh} on uniform partitions of the domain
$\Omega$ into
$N_h \in \{3\times 3,6\times 6,12\times 12,24\times 24,48\times 48\}$ squares
with mesh sizes
$h = \{3^{-1}, 3^{-1}2^{-1}, 3^{-1}2^{-2}, 3^{-1}2^{-3}, 3^{-1}2^{-4} \}$.

We now consider the approximate lower bounds that are obtained by solving
\eqref{eq:mcchh} for the aforementioned values of $h$.
In particular, we assess their approximation
quality and the running times that are required to compute them
by means of the OBBT procedure for the different values of $h$.
To obtain a valid lower bound on \eqref{eq:ocp}, we need to subtract an a priori
bound. Since the a priori analysis for this PDE is beyond of this work, we omit
this step but note that all bounds are already valid lower bounds (see below).

\sisetup{scientific-notation=true,round-mode=places,round-precision=3}
Generally, the increase in compute times for the bounds is worse than linear with 
decreasing values of $h$. The runtime for the OBBT procedure is 
$\num{28}$ seconds for $h = 3^{-1}$ and increases to $\num{95522}$
seconds for $h = 3^{-1}2^{-4}$ for \eqref{eq:mcchh} compared with a pointwise bound
computation with $\num{400783}$ seconds for \eqref{eq:mcc}
(mesh size $h = 3^{-1}2^{-5}$).

Next, we assess the approximation quality of the bounds
induced by the solutions to \eqref{eq:mcchh}
with and without bound tightening.
After 8 rounds of OBBT, the quality of the bounds does not differ much for
$h \le 3^{-1}2^{-1}$ and are close to the pointwise/baseline counterpart of 
\eqref{eq:mcc} so that a similar bound quality can be achieved with 127\,s
of compute time as with 95522\,s and 400780\,s of compute time.
Since they are also always lower, they are actually true lower bounds
for the baseline. Since we have used very conservative initial guesses
to prevent initializing from bounds that cut off feasible points,
the quality without bound tightening is extremely bad.
It is beneficial to make several rounds of bound tightening as is demonstrated
by the quality improvement of the bounds from the 7th to the 8th round.
All running times of the OBBT algorithm and the induced achieved optimal
objective values for \eqref{eq:mcchh} and \eqref{eq:mcc} are tabulated in 
\cref{tbl:2d_obbt_obj_time}.

\sisetup{scientific-notation=true,round-mode=places,round-precision=4}
\begin{table}[t]
	\caption{Running times (in seconds) for the OBBT algorithm
		and optimal objective values
		for different values of $h$
		of \eqref{eq:mcchh}
		and \eqref{eq:mcc} (mesh size $h = 3^{-1}2^{-5}$).
	}\label{tbl:2d_obbt_obj_time}
	\centering
	\begin{adjustbox}{width=\textwidth}	
		\begin{tabular}{r|cccccc}
			\toprule
			& \multicolumn{5}{c}{\eqref{eq:mcchh} with $h = $} &
			\eqref{eq:mcc} \\
			& $3^{-1}$ & $3^{-1}2^{-1}$ & $3^{-1}2^{-2}$ & $3^{-1}2^{-3}$ & $3^{-1}2^{-4}$ &  \\
			\midrule 
			Time (8 OBBT rounds) [s]
			& \num{28}
			& \num{127}
			& \num{790}
			& \num{6489}
			& \num{95522}
			& \num{400783} \\
			Objective (8 OBBT rounds)
			& \num{3.0002}
			& \num{3.1314} 
			& \num{3.1311}
			& \num{3.1301}
			& \num{3.1298}
			& \num{3.1347}\\
			Objective (7 OBBT rounds)
			& \num{2.5795}
			& \num{2.7020} 
			& \num{2.6967}
			& \num{2.6923}
			& \num{2.6912}
			& \num{2.6999}\\
			Objective (0 OBBT rounds)
			& \num{1.2734e+00}
			& \num{4.8655e-02} 
			& \num{6.5599e-04}
			& \num{2.7435e-06}
			& \num{8.3494e-09}
			& \num{3.025491e-11}\\			
			\bottomrule
		\end{tabular}
	\end{adjustbox}
\end{table}

\sisetup{scientific-notation=true,round-mode=places,round-precision=3}
\section{Conclusion}\label{sec:conclusion}
We show how to replace nonlinearities that are given in a pointwise fashion in the PDE of an
(integer) optimal control problem on the example of bilinear terms by analyzing the idea of McCormick envelopes and relaxations in the infinite-dimensional setting. To keep the computational effort manageable, in particular with respect to a bound-tightening procedure that tightens the convex relaxations, we introduce a two-level approximation scheme by means of a grid that decomposes the computational domain, on which the inequalities that yield the convex relaxation are averaged.

Our computational experiments for the 1D and the 2D example validate that the computational effort of such
a procedure can indeed be significantly reduced by using coarser grids (two to three orders of magnitude for
the 2D example). While an a posteriori comparison with the baseline solution shows that the approximation quality
is very good, the a priori approximation quality of the lower bound only improves over trivial bounds when
the mesh size is already quite fine and the involved constants can be estimated well. 
For the one-dimensional test case, the OBBT procedure for $h = 2^{-9}$ yields bounds with (relative) gaps
to the upper bounds of \num{0.06337786913326482e2}\,\% and \num{0.08549350995394171e2}\,\% in the
NLP and MINLP case compared with theoretically optimal bounds
0.1542\,\%
and \num{0.02237120424479247e2}\,\% at \num{0.07444607857395602e2}\,\% of the computational
cost for the pointwise OBBT procedure.

Consequently, in order to use the McCormick relaxations together with a priori estimates in a 
branch-and-bound procedure as sketched in \cref{rem:branch_and_bound},
it is necessary to obtain estimates on $\|u - u_h\|_H$ of higher order, and a 
deliberate analysis of the underlying PDE and its discretization are crucial when 
applying this method. Similarly, good estimates for the constant $L_u$ in 
\eqref{eq:mcchh_approx_lb_ocp} are important, too.

As mentioned before, we have used a priori error estimates in our analysis so far.
Because the actual optimal objective values for \eqref{eq:mcchh} were very close to a true lower bound on
\eqref{eq:ocp} for relatively large values of $h$ (coarser local averaging) both in 1D and 2D, we are
convinced that the analysis and integration of  \emph{a posteriori error estimates} into the procedure is
key for future research to be able to use relatively coarse grids for \eqref{eq:mcchh} and scale this
methodology.

One issue that might arise in practice, in particular for PDEs defined on multidimensional
domains, is that $L^\infty$-bounds as asserted in \eqref{eq:H10embedLinf} for our
example might not readily be available. However, we believe that additional interior regularity
(see, e.g., Theorem 9.51 in \cite{renardy2006introduction}) and nonstandard
regularity theory (see, e.g., \cite{groger1989aw}) can help establish such bounds.

Although this goes significantly beyond the scope of this article, we note that
the McCormick relaxations suggest defining a branch-and-bound algorithm in function
space in the spirit of the recent article \cite{buchheim2024parabolicIII}, 
where---depending on current fixations and the overlapping of the approximate
upper and lower bounds---one  refines the mesh that is used for the locally averaged
McCormick relaxations and the control ansatz adaptively until an acceptable gap is reached.

For more general nonlinearities than $uw$, it may be possible to combine the local averaging on the computational
domain with the successive refinement procedure (\emph{nested intervals}) of the parameter space for factorable
functions that is analyzed in \cite{scott2011generalized}, see in particular section 7 therein for an application
to obtain relaxations of ODEs.

\section*{Acknowledgments}
The authors thank Joachim St\"{o}ckler (TU Dortmund) for the pointer to the proof of 
\eqref{eq:H10embedLinf} in \cite{schmidt1940ungleichung}. The authors thank
Mariia Pokotylo (TU Dortmund) as well as two anonymous referees for helpful comments
on the manuscript.
The authors gratefully acknowledge computing time on the
LiDO3 HPC cluster at TU Dortmund, partially funded in the Large-Scale Equipment
796 Initiative by the Deutsche Forschungsgemeinschaft (DFG) as project 271512359.
Paul Manns acknowledges funding by Deutsche Forschungsgemeinschaft (DFG) under project no.\ 540198933.
This work was also supported by the U.S.~ Department of Energy, Office of Science, Office of Advanced Scientific Computing Research, Scientific Discovery through the
Advanced Computing (SciDAC) Program through the FASTMath Institute under Contract No. DE-AC02-06CH11357.

\appendix

\section{Auxiliary results}
\begin{lemma}\label{lem:lower_bound_for_weak_limit}
Let $f^n \to f$ in $L^1(\Omega)$. Let $g^n \weakto g$ in $L^1(\Omega)$.
If $f^n \le g^n$ (or $f^n \ge g^n$) holds pointwise a.e.\ for all $n \in \N$.
Then $f \le g$ (or $f \ge g$, respectively).
\end{lemma}
\begin{proof}
We  prove only the case for the $\le$-inequalities because the other case follows
by symmetry of the argument.
By way of contradiction, we assume that there exist a measurable set $A \subset \Omega$,
$\varepsilon_1 > 0$, and $\varepsilon_2 > 0$ such that $|A| > \varepsilon_1$ and
\[ f(x) > g(x) + \varepsilon_2 \text{ for a.a.\ } x \in A. \]
Egorov's theorem gives that there exists a measurable set $B \subset A$ and $n_0 \in \N$
such that $|B| > \frac{\varepsilon_1}{2}$, and for all $n \ge n_0$, we obtain
\[ f^n(x) > g(x) + \frac{\varepsilon_2}{2} \text{ for a.a.\ } x \in B. \]
Integrating over $B$ yields
\[ \int_B f^n(x) > \int_B g(x) + |B| \frac{\varepsilon_2}{2}. \]
for all $n \ge n_0$. Because $g^n \weakto g$, there exists $n_1 \ge n_0$ such that
\[ \int_B f^n(x) > \int_B g^n(x) + |B| \frac{\varepsilon_2}{4} \]
holds for all $n \ge n_1$. Consequently, there must exist a measurable subset $C \subset B$
of strictly positive measure that $f^n(x) > g^n(x)$ holds for $x \in C$,
which contradicts the  assumption that $f^n \le g^n$ holds pointwise a.e.
\end{proof}

\bibliographystyle{plain}
\bibliography{references}

\bigskip
\framebox{\parbox{.92\linewidth}{The submitted manuscript has been created by
UChicago Argonne, LLC, Operator of Argonne National Laboratory (``Argonne'').
Argonne, a U.S.\ Department of Energy Office of Science laboratory, is operated
under Contract No.\ DE-AC02-06CH11357.  The U.S.\ Government retains for itself,
and others acting on its behalf, a paid-up nonexclusive, irrevocable worldwide
license in said article to reproduce, prepare derivative works, distribute
copies to the public, and perform publicly and display publicly, by or on
behalf of the Government.  The Department of Energy will provide public access
to these results of federally sponsored research in accordance with the DOE
Public Access Plan \url{http://energy.gov/downloads/doe-public-access-plan}.}}
\end{document}